\documentclass[11pt]{amsart}

\usepackage{amssymb}
\usepackage{amsmath}
\usepackage{amscd}
\usepackage{curves}
\usepackage{epsfig}
\usepackage{bbm}
\usepackage{geometry}
\usepackage{graphicx}
\usepackage{pstricks}
\usepackage{pstricks-add}
\usepackage{pst-grad}
\usepackage{pst-plot}
\usepackage{verbatim}
\usepackage{latexsym}
\usepackage{eucal}
\usepackage{amsfonts}
\usepackage{enumerate}
\numberwithin{equation}{section}
\newtheorem{theorem}{Theorem}[section]
\newtheorem{lemma}[theorem]{Lemma}
\newtheorem{prop}[theorem]{Proposition}

\def \bpf {\begin{proof}}
\def \epf {\end{proof}}
\def \beq {\begin{equation*}}
\def \eeq {\end{equation*}}
\def \bsp{\begin{split}}
\def \esp{\end{split}}

\def \Lac {L^2_{\operatorname{ac}}(X)}

\def \CI {{C^\infty}}

\def \mca {{\mathcal A}}

\def \mcd {{\mathcal D}}
\def \mce {{\mathcal E}}
\def \mcf {{\mathcal F}}

\def \mcl {{\mathcal L}}
\def \mcm {{\mathcal M}}
\def \mcw {{\mathcal W}}

\def \mcp {{\mathcal P}}

\def \mcr {{\mathcal R}}
\def \mcs {{\mathcal S}}

\def \mcv {{\mathcal V}}
\def \mcz {{\mathcal Z}}

\def \txi {\widetilde{\xi}}

\def \mbc {{\mathbb C}}
\def \mh {{\mathbb H}}

\def \mbr {{\mathbb R}}
\def \mr {{\mathbb R}}
\def \mn {{\mathbb N}}

\def\ha {\frac{1}{2}}
\def \oq {\frac{1}{4}}
\def \tq {\frac{3}{4}}
\def \fq {\frac{5}{4}}

\def \nsq {\frac{n^2}{4}}

\def \intx {\mathring{X}}

\def \ff {{\operatorname{ff}}}

\def \dvol {\operatorname{dvol}}

\def \im {\operatorname{Im}}

\def \diag{\operatorname{Diag}}
\def \odiag{\overline{\operatorname{Diag}}}

\def \mrn {{\mathbb R}^n}
\def \mcn {\mathcal{N}}

\def \eps {\varepsilon}   
\def \vphi {\varphi}   
\def \la {\lambda}   
\def \La {\Lambda}

\def \del {\delta}   

\def \p {\partial}
\def \ps {\partial_s}

\def \novt {\frac{n}{2}}
\def \xo {{X\times_0 X}}
\def \bxo {{\p X \times_0 \p X}}

\def \beqq {\begin{equation}}
\def \eeqq {\end{equation}}

\def \mck {\mathcal{K}}

\def \mc {\mbc}

\def \wtla  {\widetilde{\Lambda}}

\def \one {\mathbbm{1}}
\numberwithin{equation}{section}

\begin{document}
\title[The Scattering Operator on  AHM]{The Scattering Operator on Asymptotically Hyperbolic Manifolds}
\author{Ant\^onio S\'a Barreto and Yiran Wang}
%\email{sabarre@math.purdue.edu, yrwang.math@gmail.com}
\address{Ant\^onio S\'a Barreto\newline
\indent Department of Mathematics, Purdue University \newline
\indent 150 North University Street, West Lafayette Indiana,  47907, USA}
\email{sabarre@math.purdue.edu}
\address{Yiran Wang \newline \indent University of Washington, Department of Mathematics  \newline \indent Box 354350  Seattle, WA 98195-4350 \newline \indent and \newline \indent Institute for Advanced Study, The Hong Kong University of Science and Technology\newline \indent Lo Ka Chung Building, Lee Shau Kee Campus, Clear Water Bay, Kowloon, Hong Kong}
\email{wangy257@math.washington.edu }
%\dedicatory{\today}
\begin{abstract}  We obtain a formula for the Schwartz kernel of the scattering operator in terms of the Schwartz kernel of the fundamental solution of the wave operator on asymptotically hyperbolic manifolds.  If there are no trapped geodesics, this formula is used  to show that the scattering operator is  a Fourier integral operator that quantizes the scattering relation.   
\end{abstract}
\maketitle

%===============================================================================%
\section{Introduction}

The main  purpose of  this paper is to study the global microlocal nature of the scattering operator on  asymptotically hyperbolic manifolds, which we shall denote by  AHM.  We will use properties of the resolvent of the Laplace operator  on AHM to analyze the asymptotic behavior of solutions of the wave equation,  more specifically the radiation fields and the scattering operator, after Friedlander \cite{fried0,fried1}.  We prove three novel results:  First we obtain a formula for the Schwartz kernel of the scattering operator in terms of the kernel of the fundamental solution of the wave operator--no assumptions about trapping are necessary, see Theorem \ref{comp-red}.  We then restrict ourselves to the class of non-trapping manifolds (recall  that a complete Riemannian manifold  is non-trapping if any maximally extended geodesic leaves any compact subset in finite time and in both directions of the curve, in particular there are no closed geodesics) and we use Theorem \ref{comp-red}  and the microlocal structure of the Schwartz kernel of the fundamental solution of the wave operator to define the scattering relation on non-trapping  AHM and to prove  that the scattering operator on non-trapping AHM is a Fourier integral operator of an appropriate class which quantizes the scattering relation.  S\'a Barreto and Wunsch \cite{SW} studied the Schwartz kernel of the radiation fields acting on compactly supported functions for non-trapping asymptotically Euclidean and non-trapping asymptotically hyperbolic manifolds,  and showed that they  are Lagrangian distributions with respect to the sojourn relation. Our third result gives  a uniform description of the Schwartz kernel of the radiation fields on non-trapping AHM up to infinity, which refines the result of \cite{SW} for AHM.

  There  is a long series of papers on scattering theory  on AHM starting with the work of  Fadeev, Fadeev \& Pavlov and Lax \& Phillips \cite{Fa,FaPa,LP,lp0,lax}.    Agmon \cite{Agmon1}, Guillemin \cite{victor}, and Perry \cite{Perry1,Perry2} also studied scattering on hyperbolic quotients.   Mazzeo \& Melrose \cite{MM} constructed a parametrix for the resolvent for the Laplacian on general AHM and  used it to show that the resolvent continues meromorphically to $\mc,$ with the exception of a discrete set of points. Guillarmou showed that the points excluded in the meromorphic continuation of the resolvent by Mazzeo and Melrose  can in fact be essential singularities, unless additional assumptions are imposed on the metric.   Vasy \cite{Vasy} has given a new proof of the meromorphic continuation of the resolvent without constructing a parametrix  for metrics that satisfy the conditions imposed by Guillarmou. Scattering theory on AHM  was studied by Borthwick and Perry \cite{BoPe}, Guillop\'e \cite{Guillope}, Guillop\'e and Zworski \cite{GZ}, Melrose \cite{M},  Joshi and S\'a Barreto \cite{JS1} and Graham and Zworski \cite{GrZw}. S\'a Barreto \cite{sbhrf} studied the Friedlander radiation fields and  the scattering operator on AHM and proved that the scattering matrix  can be obtained from the scattering operator by conjugation with the Fourier transform.  He also studied the inverse problem and proved that the scattering operator determines the manifold (including its topology and $C^\infty$ structure) and the metric up to isometries that fix the boundary.  Hora and S\'a Barreto \cite {HoSa} showed that the scattering operator restricted to an open subset of the boundary determines the manifold and the metric up to isometries that fix the open subset where the scattering operator was defined. Isozaki and Kurylev \cite{IK} have also studied scattering and inverse scattering on AHM.
  
  Melrose, S\'a Barreto and Vasy \cite{MSV}, Chen and Hassell \cite{ChenHa} and Wang \cite{Wang} studied the semiclassical resolvent on AHM.   S\'a Barreto and Wang \cite{SaWang} studied the  semiclassical resolvent and the semiclassical scattering matrix on AHM and on conformally compact manifolds with variable curvature at infinity, and showed that the semiclassical  scattering matrix is a Fourier integral operator associated to the semiclassical scattering relation (which can be obtained from the scattering relation defined in this paper by setting $\sigma=-1$ and projecting in the $s$ variable).

If $\intx$ denotes the interior of a $C^\infty$ compact manifold with boundary $X$ of dimension $n+1,$  $\rho$ is a defining function of $\p X,$ and $g$ is a $C^\infty$ metric on 
$\intx$ such that $\rho^2 g$ is smooth and non-degenerate up to 
$\p X,$ the Riemannian manifold $(\intx,g)$ is called conformally compact.   According to  Mazzeo and Melrose \cite{MM} the manifold $(\intx,g)$ is  complete and its sectional curvatures approach $-\left| d\rho|_{\p X}\right|_{h_0}^2$ as $\rho\downarrow 0$ along any curve, where  $h_0=\rho^2 g|_{\p X}.$   In the particular case when
\begin{gather}
\displaystyle \left|d\rho|_{\p X}\right|_{h_0}=1, \label{asympt-curv}
\end{gather}  
$(\intx,g)$  is said to be an asymptotically hyperbolic manifold (AHM).    This class of manifolds includes the hyperbolic space and its quotients by certain groups of symmetry, see for example \cite{Agmon1,Perry1,Perry2}.

 It follows from the definition that if $(\intx,g)$ is a conformally compact manifold,  the metric $g$  determines a conformal structure on $\p X.$ It was  shown in \cite{Gr,JS1},  that if $(\intx,g)$ is an AHM, then for each member $h_0$ of the equivalence class of $\rho^2g|_{\p X},$ where $\rho$ is a boundary defining function,  there exists a unique boundary defining function $x$ in a neighborhood  $U$ of $\p X$  and a map $\Psi: [0,\eps)\times \p X\longrightarrow U$ such that
\begin{equation}\label{prod}
\Psi^* g = \frac{dx^2}{x^2} +\frac{ h(x)}{x^2},  \;\ h(0)=h_0, \text{ on }   [0,\eps) \times \p X,
\end{equation}
where $h(x)$ is a $C^\infty$ family of Riemannian metrics on $\p X$ parametrized by $x.$

As a motivation for the definition of the scattering relation on non-trapping AHM, we recall the definition of the scattering relation for non-trapping compactly supported metric perturbations of the Euclidean space.  Suppose that $g=\sum_{i,j=1}^n g_{ij}(x) dx_idx_i$ is a $C^\infty$ non-trapping Riemannian metric on $\mrn$ and suppose that $g_{ij}(x)=\del_{ij}$ if $x\not\in K\subset \mrn,$ where $K$ is compact.  Let $B$ be a bounded ball such that $K\subset B.$   A light ray coming from $\mrn\setminus B$ enters $B$ at a point $z \in \p B$ in the direction $\zeta,$ is scattered by the metric in $K$ and goes out of $B$  at a point $z'\in \p B$ with direction $\zeta',$ the map  $(z,\zeta)\longmapsto (z',\zeta')$  is called the scattering relation, see Fig.\ref{fig6N}.  One can also take into account the  time $t$ that it takes for the geodesic to travel across $B,$ which  is called the travel (or sojourn) time.  If the geodesics are parametrized by the arc-length, the sojourn time coincide with the distance between points on the boundary.   By assumption, travel times are always finite, since the geodesics are do not get trapped inside the region. This can also be described in terms of  the submanifold  
$\La\subset T^* (\mr \times \mr^n \times \mr^n) \setminus 0$ given by
$\Lambda=\{ (t,1,z,\zeta, z',\zeta'): (z,\zeta)=\exp(t H_q)(z',\zeta')\},$ where $q=\sum_{i,j=1}^n g^{ij}\xi_i\xi_j,$ $g^{-1}=(g^{ij})$ is the dual metric to $g.$  If one takes the time into account, the scattering relation is then  re-defined to be 
\begin{gather}
\mcs_R(B)= \Lambda\cap(T^*\mr \times T_{\p B}^* \mrn \times T_{\p B}^* \mrn).\label{scat-rel-rn}
\end{gather}
 The scattering relation is intrinsically related to the Dirichlet-to-Neumann Map (DNM) for the wave equation.  If $u(t,z)$ satisfies 
 \begin{gather*}
 (D_t^2-\Delta_g) u(t,z)=0 \text{ in } \mr \times B, \\
 u(0,z)=0, \p_t u(0,z)=0, \;\ u|_{\mr \times \p B}=f(t,z),
 \end{gather*}
 the  DNM  for the wave equation is the map
 \begin{gather*}
  C^\infty(\mr \times \p B) \longrightarrow  C^\infty(\mr \times \p B) \\
 f \longmapsto \p_\nu u|_{\mr \times \p B},
 \end{gather*}
  where $\p_\nu$ denotes the normal derivative with respect to the metric $g.$  Sylvester and Uhlmann \cite{gunther2}  showed that the  DNM for the wave equation determines the scattering relation on a manifold with boundary without  conjugate points, and Uhlmann \cite{gunther1} removed the assumptions on non-existence of caustics.   Uhlmann, Pestov and Uhlmann, Stefanov and Uhlmann \cite{pestov-uhlmann,steuhl1,steuhl2,steuhl3} studied the lens rigidity and boundary rigidity inverse problems, where one wants to obtain information about the manifold from its scattering relation.   In this article we show the analogue of Uhlmann's result for the scattering operator on non-trapping AHM and we also build a framework which makes it possible to pose the lens rigidity question for AHM.
 \begin{figure}
% Generated with LaTeXDraw 2.0.8
% Tue Oct 27 11:18:31 EDT 2015
% \usepackage[usenames,dvipsnames]{pstricks}
% \usepackage{epsfig}
% \usepackage{pst-grad} % For gradients
% \usepackage{pst-plot} % For axes
\scalebox{.7} % Change this value to rescale the drawing.
{
\begin{pspicture}(0,-4.0482655)(12.461016,2.517555)
\psbezier[linewidth=0.04](0.9410156,0.37915665)(0.32046875,-0.4050127)(1.3198191,-0.8474328)(2.2210157,-1.2808434)(3.1222122,-1.7142539)(4.3186464,-2.2201526)(4.5210156,-1.2408433)(4.723385,-0.26153404)(1.5615625,1.163326)(0.9410156,0.37915665)
\pscircle[linewidth=0.04,dimen=outer](2.7310157,-0.63084334){2.71}
\psdots[dotsize=0.12](4.861016,1.0591567)
\usefont{T1}{ptm}{m}{n}
\rput(4.792471,0.64415663){$z'$}
\psdots[dotsize=0.12](1.6810156,1.8591566)
\usefont{T1}{ptm}{m}{n}
\rput(1.5924706,1.4041567){$z'$}
\psdots[dotsize=0.12](3.9210157,-3.0408432)
\usefont{T1}{ptm}{m}{n}
\rput(3.7624707,-3.5558434){$z$}
\usefont{T1}{ptm}{m}{n}
\rput(3.0124707,-0.7158434){$K$}
\usefont{T1}{ptm}{m}{n}
\rput(1.2124707,-3.1758432){$\p B$}
\usefont{T1}{ptm}{m}{n}
\rput(1.9824708,-2.1558433){$B$}
\pscircle[linewidth=0.04,dimen=outer](9.751016,-0.79084337){2.71}
\psline[linewidth=0.04cm,arrowsize=0.05291667cm 2.0,arrowlength=1.4,arrowinset=0.4]{->}(9.781015,1.8991567)(9.841016,-3.5208433)
\rput{-46.93517}(2.7843168,4.331723){\psarc[linewidth=0.04,arrowsize=0.05291667cm 2.0,arrowlength=1.4,arrowinset=0.4]{->}(6.381016,-1.0408434){2.5}{0.0}{102.991585}}
\usefont{T1}{ptm}{m}{n}
\rput(7.832471,-3.2358434){$z'$}
\psdots[dotsize=0.12](7.801016,1.0391567)
\psdots[dotsize=0.12](8.061016,-2.8808434)
\psdots[dotsize=0.12](9.801016,1.8991567)
\psdots[dotsize=0.12](10.721016,1.6991566)
\psdots[dotsize=0.12](12.181016,-1.9208434)
\usefont{T1}{ptm}{m}{n}
\rput(7.662471,1.2641567){$z$}
\psdots[dotsize=0.12](9.821015,-3.5008433)
\usefont{T1}{ptm}{m}{n}
\rput(9.992471,2.3241568){$z'$}
\usefont{T1}{ptm}{m}{n}
\rput(9.822471,-3.8758433){$z$}
\usefont{T1}{ptm}{m}{n}
\rput(11.092471,2.0041566){$z'$}
\usefont{T1}{ptm}{m}{n}
\rput(12.10247,-2.6758432){$z$}
\rput{163.1795}(27.075356,-2.344776){\psarc[linewidth=0.04,arrowsize=0.05291667cm 2.0,arrowlength=1.4,arrowinset=0.4]{->}(13.711016,0.82915664){3.11}{0.0}{78.25404}}
\usefont{T1}{ptm}{m}{n}
\rput(10.9824705,-2.1158433){$\mh^{n+1}$}
\rput{-214.79297}(24.32582,-7.9032974){\psarc[linewidth=0.04,arrowsize=0.05291667cm 2.0,arrowlength=1.4,arrowinset=0.4]{->}(13.401015,-0.14084335){3.28}{0.0}{90.555504}}
\usefont{T1}{ptm}{m}{n}
\rput(11.64247,-3.1358433){$z$}
\usefont{T1}{ptm}{m}{n}
\rput(10.942471,-3.6758432){$z$}
\rput{-46.449833}(4.0191407,4.764421){\psarc[linewidth=0.04,arrowsize=0.05291667cm 2.0,arrowlength=1.4,arrowinset=0.4]{->}(7.5610156,-2.3008432){0.72}{0.0}{166.91537}}
\rput{-50.431915}(3.91151,4.8646894){\psarc[linewidth=0.04,arrowsize=0.05291667cm 2.0,arrowlength=1.4,arrowinset=0.4]{->}(7.1210155,-1.7208433){1.48}{0.0}{141.8291}}
\usefont{T1}{ptm}{m}{n}
\rput(6.8224707,-0.21584335){$z$}
\usefont{T1}{ptm}{m}{n}
\rput(6.8424706,-1.6758434){$z$}
\psline[linewidth=0.04cm](4.841016,1.0591567)(3.5610156,-0.060843352)
\psbezier[linewidth=0.04](3.5810156,-0.06693031)(3.0410156,-0.68084335)(3.2210157,-1.0408434)(3.3410156,-1.7608434)
\psline[linewidth=0.04cm,arrowsize=0.05291667cm 2.0,arrowlength=1.4,arrowinset=0.4]{->}(3.3610156,-1.7608434)(3.9410157,-3.1208434)
\psline[linewidth=0.04cm](1.7010156,1.8591566)(2.3210156,0.51915663)
\psbezier[linewidth=0.04](2.330196,0.47915664)(2.6210155,-0.14084335)(2.1610155,-0.46084335)(1.6810156,-1.0008434)
\psline[linewidth=0.04cm,arrowsize=0.05291667cm 2.0,arrowlength=1.4,arrowinset=0.4]{->}(1.6610156,-1.0008434)(0.52101564,-2.2208433)
\usefont{T1}{ptm}{m}{n}
\rput(0.2224707,-2.3758433){$z$}
\usefont{T1}{ptm}{m}{n}
\rput(3.1024706,1.1241566){$g_{ij}=\del_{ij}$}
\end{pspicture} 
}
\caption{The scattering relation for non-trapping compactly supported perturbations of the Euclidean metric and for hyperbolic space.}
\label{fig6N}
\end{figure}
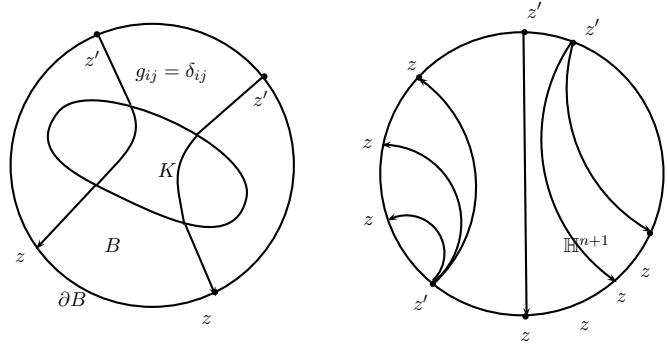

In the case of asymptotically hyperbolic manifolds, or even the hyperbolic space,  one faces several difficulties to define the scattering relation.  First, the wave operator corresponding to the metric is degenerate at the boundary,  the length of its bicharacteristics go to infinity as they approach the boundary,  the projections of the bicharacteristics, which are geodesics of the metric, always intersect $\p X$ orthogonally,  see Fig.\ref{fig6N}.   To define the analogue of the scattering relation on AHM, we first discuss the scattering operator,  which is the analogue of the DNM, defined in \cite{sbhrf},  and we begin by recalling the definitions of the radiation fields and the scattering operator from \cite{sbhrf}. Let $u(t,z)$ be the solution of
\begin{gather}
\begin{gathered}
\square u=\left(D_t^2-\Delta_g+\frac{n^2}{4} \right) u=0, \\
u(0,z)=f_1(z), \;\ \p_t u(0,z)=f_2(z), \;\   f_1, f_2 \in C_0^\infty(\intx).
\end{gathered}\label{CPWE}
\end{gather}
It was shown in \cite{sbhrf} that for  any choice of a boundary defining function $x$ such that \eqref{prod} holds, then for $z=(x,y),$  and for $s_+=t+\log x,$ and $s_-=t-\log x,$
\begin{gather}
\begin{gathered}
V_+(x,s_+,y)= x^{-\novt}  (\one_+ (t)u)(s_+- \log x, x,y) \in C^\infty( [0,\eps)_x \times \mr_{s_+} \times \p X ),\\
V_-(x,s_-,y)= x^{-\novt}  (\one_-(t) u)(s_-+ \log x, x,y) \in C^\infty( [0,\eps)_x \times \mr_{s_-} \times \p X ),
\end{gathered} \label{radf}
\end{gather}
where $\one_\pm (t)=1$ for $\pm t>0$ and $\one_\pm (t)=0$ for $\pm t<0.$ Following Friedlander \cite{fried0,fried1},  Lax \cite{lax} and  Lax and Phillips \cite{lp0}, the forward and backward radiation fields  for AHM were defined in \cite{sbhrf} as
\begin{gather}
\begin{gathered}
\mcr_+(f_1,f_2)(s_+,y)=  \p_{s_+}V_+(x,s_+,y)|_{\{x=0\}} \text{ and }  \mcr_-(f_1,f_2)(s_-,y)=  \p_{s_-}V_-(x,s_-,y)|_{\{x=0\}}.
\end{gathered}\label{radf1}
\end{gather}
Of course these operators depend on the choice of the boundary defining function $x,$ and we will pick one particular $x$ for which \eqref{prod} is satisfied.  One can modify the definition to make it independent of  the choice of $x$ by having these operators act on appropriate bundles, but we will not pursue this here.

  It was shown in \cite{sbhrf} that  the maps $\mcr_\pm$ have extensions
\begin{gather}
\begin{gathered}
\mcr_\pm: E_{ac}(X) \longrightarrow L^2(\mr \times \p X, ds \dvol_{h_0})\\
(f_1,f_2) \longmapsto \mcr_\pm(f_1,f_2),
\end{gathered} \label{radfs-0}
\end{gather}
 as isometric isometries, where  $E_{ac}(X)$ is the space of functions $(f_1,f_2)$ with finite energy which are orthogonal to the eigenfunctions of $\Delta_g,$ and  where $h_0 = x^2 g|_{\p X}$ is the metric on $\p X$ induced by $g$ and $x.$  While the restrictions $V_\pm|_{\{x=0\}}$ are well-defined, they are not necessarily  $L^2$ functions, and the reason for taking the derivative in $s_\pm$ of $V_\pm$ in the definition of $\mcr_\pm,$  is to make these maps unitary. 
 
We shall say that  an AHM $(\intx,g)$  is non-trapping if any maximally extended geodesic $\gamma(t) \rightarrow \p X$ as $\pm t\rightarrow \infty.$  S\'a Barreto and Wunsch \cite{SW}  proved  that, for non-trapping asymptotically Euclidean and asymptotically hyperbolic manifolds,  $K_{\mcr_{\pm}}(s,y,z'),$ 
 the Schwarz kernel of $\mcr_\pm,$ are Lagrangian distributions associated with the sojourn relation. Theorem \ref{radfs-th-N} below strengthens the result of \cite{SW} in the AHM case  by describing the behavior of $K_{\mcr_\pm}(s,y,z')$ as $z'\rightarrow \p X.$ 

The wave group $U(t)$ is the map
\begin{gather}
U(t)(f_1,f_2)= (u(t),\p_t u(t)), \label{wavegp}
\end{gather}
where $u(t)$ solves \eqref{CPWE}.  The operators $\mcr_\pm$ are  translation representations of $U(t)$ as in the Lax-Phillips theory \cite{lp0}, i.e.
\begin{gather}
\mcr_{\pm}(U(T)(f_1,f_2))(s_\pm,y)= \mcr_{\pm}(f_1,f_2)(s_\pm+T,y).\label{translrep}
\end{gather}

The scattering operator is defined to be the map
\begin{gather}
\begin{gathered}
\mathcal{S}:L^2(\mr_s \times\partial X)\longrightarrow L^2(\mr_s \times\partial X),\\ \mathcal{S}=\mathcal{R}_{+}\circ\mathcal{R}_{-}^{-1},
\end{gathered}\label{scatmat2}
\end{gather}
which is unitary in $L^2(\partial X\times \mathbb{R})$ and,  in view of \eqref{translrep},  commutes with translations in the $s$ variable.   Therefore, it is a convolution operator in the $s$-variable, and  there exists $\mck \in C^{-\infty}(\mr \times \p X \times \p X)$ such that the Schwartz kernel of $\mcs$ satisfies
\begin{gather}
K_{\mcs}(s,y,s',y')= \mck(s-s',y,y'). \label{kerscatmat}
\end{gather}

 Let $E_+(t,z,z')$  and $E_-(t,z,z')$ denote the Schwartz kernel of forward and backward fundamental solutions of the wave equation. In other words
 \begin{gather}
 \begin{gathered}
 (D_t^2-\Delta_g+\nsq) E_\pm (t,z,z')= \delta(z,z') \del(t), \\
 E_+(t,z,z') \text{ supported in } t\geq d_g(z,z'), \\ 
 E_-(t,z,z') \text{ supported in } t\leq - d_g(z,z'), \\ 
 \text{ where } d_g \text{ is the distance function of the metric } g.
 \end{gathered}\label{CPW}
 \end{gather}

We will first prove a formula connecting the kernel of the scattering operator to  $E_+(t,z,z'):$
\begin{theorem}\label{comp-red}  Let  $(\intx,g)$ be an AHM and let $E_+(t,z,z')$ be  defined in \eqref{CPW}.  Let  $x$ be a defining function of $\p X$ such that \eqref{prod} holds, and denote $z=(x,y)$ and $z'=(x',y').$  If $K_{\mcs}(s,y,s',y')$ is the Schwartz  kernel of  the scattering operator $\mcs,$ then 
\begin{gather}
\begin{gathered}
 K_{\mcs}(s,y,s',y')=\lim_{x\rightarrow 0} \lim_{x'\rightarrow 0} \ha (xx')^{-\novt} \p_s E_+ (s-s'-\log x- \log x',z,z'),
\end{gathered}\label{kerscat-0}
\end{gather}
in the sense that  there exists $\mcm(t,z,s',y')$ such that for $\psi\in C_0^\infty(\intx) \cap \Lac$ and $F\in L^2(\mr \times \p X),$ 
\begin{gather}
\begin{gathered}
\lim_{x'\rightarrow 0} {x'}^{-\novt} \int_{X} \p_s E_+(t-s'-\log x', z,z')  \psi(z) \dvol_{g(z)}  =  \int_{X} \mcm(t,z, s',y')  \psi(z) \dvol_{g(z)},  \\
\text{ and } 
\lim_{x\rightarrow 0} x^{-\novt} \int_{\mr \times \p X} \mcm(s-\log x, x,y, s', y') F(s',y') \; ds' \dvol_{h_0(y')}= \\ \int_{\mr \times \p X} K_{\mcs}(s,y,s',y') F(s',y') ds' \dvol_{h_0(y')}.
\end{gathered}\label{sep-limits}
\end{gather}
\end{theorem}

The analogue of \eqref{kerscat-0} for asymptotically Euclidean manifolds is conjectured by Friedlander  on page 15 of \cite{fried1}.  We prove \eqref{kerscat-0} in the AHM case, but its analogue should be true for manifolds where the radiation fields are well defined and unitary, as for example asymptotically Euclidean manifolds and asymptotically complex hyperbolic manifolds. The analogue of this formula on the frequency side was proved in \cite{GZ,JS1}, see equation \eqref{scatmat-JS} below.  Of course, the problem is to show that the Fourier transform commutes with the limits.

This formula does not say much about the microlocal  nature of $K_{\mcs}(s,y,s',y').$  In the case of non-trapping AHM we can prove that  the limit  \eqref{kerscat-0} is a Lagrangian distribution of an appropriate class on $\mr \times \p X \times \p X \times \mr,$ but before we can state our result, we need to recall the definition of the zero stretched product introduced by Mazzeo and Melrose \cite{MM}.  Let 
\begin{gather*}
\diag= \{(z,z') \in \intx\times \intx: \; z=z'\}
\end{gather*}
denote the diagonal in $\intx \times \intx.$ The closure of this submanifold meets the boundary of $X\times X$ at
\beq
\p \odiag = \{(z, z) \in \p X\times \p X\} = \odiag \cap(\p X\times \p X).
\eeq
The  $0$-stretched product $\xo$ is the blow-up of the manifold  $X\times X$ along the submanifold $\p \odiag.$   As a set, $\xo$ is given by
\beq
X\times_0 X = (X\times X)\backslash \p\odiag \sqcup S_{++}(\p \odiag),
\eeq
where $S_{++}(\p\odiag)$ denotes the inward pointing spherical bundle of $T_{\p\odiag}^*(X\times X).$   $\xo$ is then equipped with a topology and smooth structure  of a manifold with corners such that the blow-down map 
\beq
\beta_0: X\times_0 X \rightarrow X\times X
\eeq 
is $C^\infty.$ Of course, in the interior of $\xo,$ $\beta_0$ is a diffeomorphism between open $C^\infty$ manifolds.

The manifold $\xo$  has three boundary hypersurfaces, which we denote the left face  $L=\overline{\beta_0^{-1}(\p X\times \intx)},$  the right face $R=\overline{\beta_0^{-1}(\intx \times \p X)}$  and the front face $\ff= \overline{\beta_0^{-1}(\p\odiag)}.$ The lifted diagonal is denoted by $\diag_0=\overline{\beta_0^{-1}(\odiag)},$ see Fig.\ref{fig1}.  
We shall use $\rho_\bullet$  to denote a defining function of the face $\bullet=R,L,\ff.$

By abuse of notation, we will also denote

\begin{gather*}
\beta_0: \mr_t \times \xo \longrightarrow \mr_t\times X \times X \\
(t, m) \longmapsto (t, \beta_0(m)).
\end{gather*}

\begin{figure}
\centering
\scalebox{0.7}
{
\begin{pspicture}(0,-3.8529167)(17.765833,3.8729167)
\rput(3.7458334,-0.14708334){\psaxes[linewidth=0.04,labels=none,ticks=none,ticksize=0.10583333cm](0,0)(0,0)(4,4)}
\rput(0.74583334,-3.1470833){\psaxes[linewidth=0.04,labels=none,ticks=none,ticksize=0.10583333cm](0,0)(0,0)(4,4)}
\psline[linewidth=0.04cm,arrowsize=0.05291667cm 2.0,arrowlength=1.4,arrowinset=0.4]{->}(3.7458334,-0.14708334)(0.24583334,-3.6470833)
\psline[linewidth=0.04cm,dotsize=0.07055555cm 3.0]{*-}(2.2458334,-1.6470833)(4.7458334,3.1729167)
\psline[linewidth=0.04cm](0.74583334,0.85291666)(3.7458334,3.8529167)
\psline[linewidth=0.04cm](4.7458334,-3.1470833)(7.7458334,-0.14708334)
\rput(2.2458334,-1.6470833){\psaxes[linewidth=0.04,arrowsize=0.05291667cm 2.0,arrowlength=1.4,arrowinset=0.4,labels=none,ticks=none,ticksize=0.10583333cm]{->}(0,0)(0,0)(5,5)}
\usefont{T1}{ptm}{m}{n}
\rput(5.295833,2.6979167){$\odiag$}
\usefont{T1}{ptm}{m}{n}
\rput(2.6,-2){$\p\odiag$}
\usefont{T1}{ptm}{m}{n}
\rput(7.4058332,-1.8420833){$x'$}
\usefont{T1}{ptm}{m}{n}
\rput(2.0358334,3.3779166){$x$}
\usefont{T1}{ptm}{m}{n}
\rput(1.35,0.5179167){$X$}
\usefont{T1}{ptm}{m}{n}
\rput(4.1658335,-2.6820834){$X$}
\usefont{T1}{ptm}{m}{n}
\rput(0.97583336,-3.6020834){$y - y'$}
\rput(13.745833,-0.14708334){\psaxes[linewidth=0.04,labels=none,ticks=none](0,0)(0,0)(4,4)}
\rput(10.745833,-3.1470833){\psaxes[linewidth=0.04,labels=none,ticks=none](0,0)(0,0)(4,4)}
\psline[linewidth=0.04cm](13.745833,3.8529167)(10.745833,0.85291666)
\psline[linewidth=0.04cm](17.745832,-0.14708334)(14.745833,-3.1470833)
\psline[linewidth=0.04cm](11.865833,-2.0270834)(10.745833,-3.1470833)
\psarc[linewidth=0.04](13.425834,-2.0270834){1.56}{90.0}{180.0}
\psarc[linewidth=0.04](11.885834,-0.50708336){1.54}{268.5312}{1.27303}
\psline[linewidth=0.04cm](13.745833,-0.14708334)(13.425834,-0.48708335)
\psarc[linewidth=0.04](12.155833,-1.9770833){1.03}{28.855661}{82.11686}
\psline[linewidth=0.04cm,dotsize=0.07055555cm 3.0]{*-}(12.765833,-1.1470833)(14.725833,3.1129167)
\usefont{T1}{ptm}{m}{n}
\rput(15.245833,2.6579165){$\diag_0$}
\usefont{T1}{ptm}{m}{n}
\rput(14.445833,-1.9970833){\large $L$}
\usefont{T1}{ptm}{m}{n}
\rput(11.635834,0.8229167){\large $R$}
\usefont{T1}{ptm}{m}{n}
\rput(12.325833,-1.6020833){$\ff$}
\psline[linewidth=0.04cm,arrowsize=0.05291667cm 2.26,arrowlength=1.4,arrowinset=0.4]{->}(10.1258335,3.4129167)(8.105833,3.4129167)
\usefont{T1}{ptm}{m}{n}
\rput(9.155833,2.9029167){\large $\beta_0$}
\end{pspicture}
}
\caption{The $0$-blown-up space $X\times_0X$.}
\label{fig1}
\end{figure}
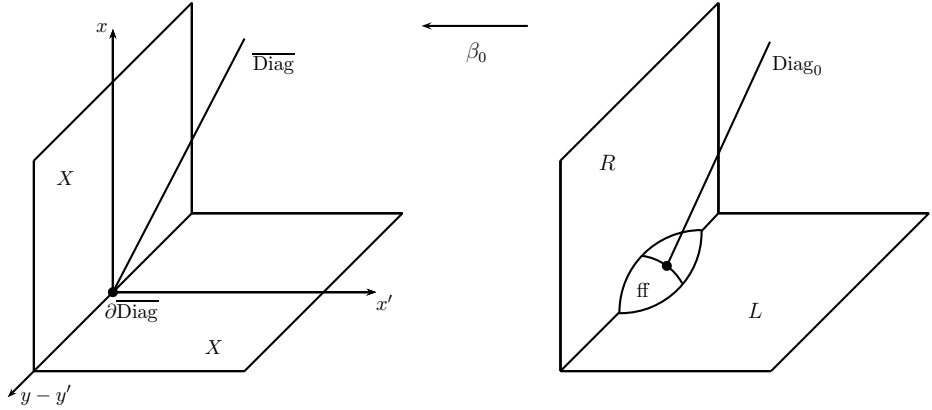

  We will work with the projection of $\xo$ to $\{\rho_L=0\}$  and to $\{\rho_R=0\},$ and  as observed in \cite{GrZw,GZ,JS1}, these define manifolds with corners $\p X\times_0 X$  and $ X \times_0 \p X$ obtained by blowing up $\p X\times \p X$  and $X \times \p X $ along the manifolds
\begin{gather*}
\p_L \odiag=\{(y,z'): y\in \p X, z'\in X, \; z'=(x',y)\}, \;\ \p_R \odiag=\{(z,y'): y'\in \p X, z\in X, \; z=(x,y')\},
\end{gather*}
respectively, and where 
\begin{gather}
\beta_{0L}= \beta_{0}|_{\{\rho_L=0\}} \text{ and }  \beta_{0R}= \beta_{0}|_{\{\rho_R=0\}} \label{betaol}
\end{gather}
are the associated blow-down maps.   As above, by abuse of notation, we shall also denote
\begin{gather*}
\beta_{0L}: \mr \times \p\xo \longrightarrow \mr \times \p X \times X \\
(t, \tilde m) \longmapsto (t, \beta_{0L}(\tilde m)) \\ \text{ and } \\
\beta_{0R}: X \times_0 \p X \times \mr \longrightarrow  X \times \p X  \times \mr \\
(\tilde m, t) \longmapsto (\beta_{0R}(\tilde m), t).
\end{gather*}

As also observed in \cite{GrZw,GZ,JS1}, the projection of $\xo$ to $\{\rho_R=\rho_L=0\}$ defines the manifold with boundary $\p X \times_0 \p X$ obtained by blowing $\p X \times \p X$ along its diagonal 
\begin{gather*}
\diag_\p=\{ (y,y') \in \p X \times \p X: \; y=y'\}.
\end{gather*}
 Again, by abuse of notation, we will use $\beta_{\p}$ to denote either of the blow-down maps
\begin{gather}
\begin{gathered}
\beta_{\p}: \p X \times_0 \p X \longrightarrow \p X \times \p X\\
\tilde m \longmapsto \beta_{\p}(\tilde m) \\
\text{ and } \\
\beta_{\p}: \mr_s \times  \bxo  \times \mr_s \longrightarrow \mr_s \times \p X \times \p X \times \mr_s \\
(s,\tilde m,s') \longmapsto (s,\beta_\p(\tilde m), s').
\end{gathered}  \label{defbbdry} 
\end{gather}

Now we can state our second result:
\begin{theorem}\label{scat-FIO}  Let $(\intx,g)$ be a non-trapping AHM. Fix a defining function $x$ of $\p X$ for which \eqref{prod} holds and let $K_{\mcs}$ be the Schwartz kernel of the corresponding  scattering operator.   Then
\begin{gather}
\begin{gathered}
\beta_{\p}^* K_{\mcs}(s,\tilde m,s')= \mca(s-s'-2\log \rho_{\ff_0},\tilde m), \\
\mca(s,\tilde m)  \in I^{\oq}(\mr\times \p X \times_0 \p X, \La_{\p+}^f,\Omega_{\mr\times \p X \times_0 \p X}^\ha)+  I^{\oq}(\mr\times \p X \times_0 \p X, \La_{\p-}^f,\Omega_{\mr\times \p X \times_0 \p X}^\ha),
\end{gathered}\label{kerscatmat1}
\end{gather}
where $\beta_{\p}$ is the map defined in \eqref{defbbdry},  $\Lambda_{\p\pm}^f$ are the Lagrangian submanifolds defined in \eqref{bdryla} and $\rho_{\ff_0}$ is a defining function of the boundary of $\p X \times_0 \p X.$ 
\end{theorem}

In view of Theorem \ref{scat-FIO} we shall say that
\begin{gather}
\La_\p^f=\La_{\p+}^f\cup \La_{\p-}^f \text{ is the scattering relation of the non-trapping AHM } (\intx,g). \label{def-scat-rel0}
\end{gather}

One could just as well have defined  $\La_{\p+}^f \cap \{\sigma=-1\}$ to be the scattering relation.  As we will see in Section \ref{ULM}, $\La_\p^f$ is  foliated by $\La_{\p}^f \cap \{\sigma= \text{constant} \}$ and the leaf $\La_{\p+}^f \cap \{\sigma=-1\}$ is the one associated with unit speed geodesics, which is consistent with the definition \eqref{scat-rel-rn} of the scattering relation for metric perturbations mentioned above, and in this case the variable $s$ plays the role of time. This is called the sojourn time, and is related to the Busemann function  used in differential geometry. This will be made more clear in Section \ref{ULM}.

The reader might think that \eqref{kerscatmat1} contradicts the fact that $\mcs$ is bounded on $L^2(\mr\times \p X),$ but this is explained by the following 
\begin{prop}  Let  $X$ be a $C^\infty$ manifold  of dimension $d$ and  let $\La \subset (T^*\mr\setminus 0 ) \times (T^* X \setminus 0)$  be a $C^\infty$ conic closed Lagrangian submanifold.  Let 
$A(s,y) \in I^m( \mr \times X, \La, \Omega^\ha(\mr \times X)).$ Then 
\begin{gather*}
A(s-s', y)\in I^{m-\oq}( \mr \times X \times \mr, \widetilde{\La}, \Omega^\ha(\mr \times X\times \mr)), \text{ where } \\
\widetilde{\La}=\{ (s,\sigma,y,\eta, s',\sigma'): (s-s', y,\eta)\in \La, \; \sigma=-\sigma'\}.
\end{gather*}
\end{prop}
\begin{proof}  Using the notation from  \cite{Hormander}, $A(s,y)$ is  microlocally given by an oscillatory integral
\begin{gather*}
A(s,y)= \int_{\mr^N} e^{i\Phi(s,y,\theta)} a(s,y,\theta) \; d\theta, \;\ a\in S^{m+\frac{(d+1-2N)}{4}}(\mr \times \mr^{d}\times \mr^N).
\end{gather*}
where  $\Phi(s,y,\theta)$ locally parametrizes $\La$ in the sense that
\begin{gather*}
\La= \{ (s,y, d_s \Phi, d_y \Phi): \; d_\theta \Phi(s,y,\theta)=0\}, \text{ and } \\
d_{s,y,\theta}(d_\theta\Phi(s,y,\theta)) \text{ are linearly independent when } d_\theta \Phi(s,y,\theta)=0.
\end{gather*}
If we define $\Psi(s,y,s',\theta)=\Phi(s-s',y,\theta),$ then
\begin{gather*}
d_{s,y,s',\theta}(d_\theta\Psi(s,y,s',\theta)) \text{ are linearly independent when } d_\theta \Psi(s,s',y,\theta)=0, \\
\text{ and } \\
\{(s,s',y, d_s \Psi, d_{s'}\Psi, d_y\Psi: d_\theta\Psi=0\}=\{(s-s',y,\eta)\in \La, \sigma=-\sigma'\}=\wtla.
\end{gather*}
Therefore $\wtla$ is a Lagrangian submanifold, $\Psi(s,s',y,\theta)$ parametrizes $\wtla$ and 
\begin{gather*}
A(s-s',y)= \int_{\mr^N} e^{i\Psi(s,y,s',\theta)} a(s-s',y,\theta) \; d\theta, \text{ where } \\
a(s-s',y,\theta) \in S^{m-\oq +\frac{(d+2-2N)}{4}}(\mr\times \mr^d \times \mr \times \mr^N).
\end{gather*}
So we conclude that $A(s-s', y)\in I^{m-\oq}( \mr \times X \times \mr, \widetilde{\La}, \Omega^\ha(\mr \times X\times \mr)).$
\end{proof}
Therefore, except for the singular term $-2\log \rho_{\ff_0},$ $\mcs$ is a  Fourier integral operator (FIO) of oder zero.

\section{ The scattering operator and the wave group}

   The goal of this section is to prove Theorem \ref{comp-red}, and we will rely on spectral methods based on the work of  Guillop\'e \cite{Guillope},  Joshi and S\'a Barreto \cite{JS1}, Mazzeo and Melrose \cite{MM} and S\'a Barreto \cite{sbhrf}.  
   
   The Laplace operator $\Delta_g$ is a self adjoint unbounded operator on $L^2(X)$ and its domain is given by $H_0^2(X),$ where  for $m \in \mn,$
   \begin{gather}
   \begin{gathered}
   H_0^m(X)= \{ \psi\in L^2(X): \mcv_1 \mcv_2 \ldots \mcv_k \psi \in L^2(X), \; k\leq m \\
\text{ where } \mcv_j \text{ is a vector field vanishing at } \p X\}. 
\end{gathered}\label{sobolevsp}
\end{gather}

In local coordinates $(x,y)$ for which \eqref{prod} holds a vector fields vanishes at $\p X$ if and only if it is locally given by
$\mcv = a(x,y) x \p_x + \sum_{j=1}^n a_j(x,y) x\p_{y_j},$ where $a, a_j \in C^\infty.$

According to \cite{MM}, the spectrum of $\Delta_g$  is given by
\begin{gather}
\begin{gathered}
\text{ a finite point spectrum } \sigma_{pp}\subset (0, \nsq) \text{ and the continuous spectrum } \sigma_{ac}=[\nsq, \infty).
 \end{gathered}\label{spectrum}
 \end{gather}

 Then according to the spectral theorem, the resolvent
 \begin{gather}
 \begin{gathered}
R(\la)= \left(\Delta_g-\nsq-\la^2\right)^{-1}: L^2(X) \longrightarrow H_0^2(X) \\
 \text{ is a bounded  operator which is meromorphic in } \la \text{ for }  |\im \la| >0  \\
 \text{with a finite number of  poles given by } \pm i |\mu| \text{ where } -\mu^2 \text{ is an eigenvalue of } \Delta_g-\nsq \\
 \text{ and the multiplicity of the pole is equal to the multiplicity of the eigenvalue.}
\end{gathered}\label{resolv0}
\end{gather}
Let us first take the part of $R(\la)$  which is holomorphic in $\im \la<<0,$ which we shall denote by $R_+(\la).$  Mazzeo and Melrose showed that $R_+(\la)$ continues meromorphically to $\mc\setminus \{i(k+\ha), \; k\in \mn\},$  as a family of operators
\begin{gather}
R_+(\la): C_0^\infty( \intx) \longrightarrow C^\infty(\intx), \label{domain-res}
\end{gather}
and Guillarmou showed that generically the points $i(k+\ha)$ are  essential singularities of  $R_+(\la)$ unless the family of metrics on $\p X$ denoted by $h(x)$ in \eqref{prod} is a function of $x^2.$ We shall call $R_+(\la)$ the forward resolvent.

Similarly, if we say that $R_-(\la)$ is the part of $R(\la)$ which is holomorphic in $\im \la>>0,$ the result  of Mazzeo and Melrose  guarantees that $R_-(\la)$  continues meromorphically to $\mc\setminus\{-i(k+\ha), \; k\in \mn\}$ as a family of operators satisfying \eqref{domain-res}.  We shall call $R_-(\la)$ the backward resolvent.

Let $u(t,z)$ satisfy \eqref{CPWE} with initial data $(\vphi,\psi)\in E_{ac}(X),$ $\vphi,\psi\in C_0^\infty(\intx),$ then $u_+(t,z)=\one_+(t)u (t,z)$ and $u_-(t,z)=\one_-(t)u(t,z)$ satisfy
\begin{gather} 
\begin{gathered}  
(D_t^2 -\Delta_g+\nsq) u_+(t,z)= -(\psi \delta(t)+ \vphi \delta'(t)), \\
(D_t^2 -\Delta_g+\nsq) u_-(t,z)= (\psi \delta(t)+ \vphi \delta'(t)),
\end{gathered}\label{cut-off1}
\end{gather}
 If $E_+$ and $E_-$ are the forward and backward fundamental solutions of the wave operator, then
\begin{gather}
\begin{gathered}
u_+(t,z)= - \int_{X} \left( E_+(t,z,z')\psi(z')+ \p_t E_+(t,z,z') \vphi(z')\right) \dvol_{g(z')}, \\
u_-(t,z)= \int_{X} \left( E_-(t,z,z')\psi(z')+ \p_t E_-(t,z,z') \vphi(z')\right) \dvol_{g(z')}.
\end{gathered} \label{sol-CP}
\end{gather}

Since $u(t,z)$ has finite energy, $u_+(t,z)$ is tempered in $t,$ and hence
\begin{gather*}
\left(\Delta_g-\nsq-\la^2\right) \widehat{u_+}(\la,z)= \psi(z)+i\la \vphi(z)
\end{gather*}
where $\widehat{u_+}(\la,z)$ is the Fourier transform in $t$ of $u_+(t,z).$ And since $u_+(t,z)=0$ for $t<0,$ $\widehat{u_+}(\la,z)$ is holomorphic in $\la$ for $\im\la<0.$  It follows that
\begin{gather*}
\widehat{u_+}(\la,z)= R_+(\la)( \psi(z)+i\la \vphi(z)), \;\ \im\la<0, \;\ \la^2\not\in \{ -\mu^2\in  \sigma_{pp}(\Delta_g-\nsq)\}.
\end{gather*}

We conclude from \eqref{sol-CP}  that if $R_+(\la,z,z')$ denotes the Schwartz kernel of $R_+(\la),$ then
\begin{gather}
R_+(\la,z,z')= \widehat{E_+}(\la,z,z'), \; \im \la<0. \label{res-fw}
\end{gather}

One can do the analogue construction for the backward fundamental solution, and the backward resolvent. Namely, we take the Fourier transform in $t$ of  $u_-(t,z),$ and since $u_-(t,z)=0$ in $t>0,$ $\widehat{u_-}(\la,z)$ is holomorphic in $\im \la>0,$  we obtain
\begin{gather*}
\left(\Delta_g-\nsq-\la^2\right) \widehat{u_-}(\la,z)= -(\psi(z)+i\la\vphi(z)),
\end{gather*}
and therefore 
\begin{gather*}
\widehat{u_-}(\la,z)=-R_-(\la) (\psi(z)+i\la \vphi(z)), \;\ \im \la>0, \;\ \la^2\not\in \{ -\mu^2\in  \sigma_{pp}(\Delta_g-\nsq)\}.
\end{gather*}
  If $E_-(t,z,z')$ is Schwartz  kernel of the backward fundamental solution of the wave equation defined in \eqref{CPW} and if $R_-(\la,z,z')$ denotes the Schwartz kernel of $R_-(\la),$ then
\begin{gather}
R_-(\la,z,z')= \widehat{E_-}(\la,z,z'), \; \im \la>0. \label{res-bw}
\end{gather}

We set $s_+=t+\log x,$  $s_-=t-\log x,$ and define
\begin{gather}
\begin{gathered}   
\mce_+(s_+,y,z') =  x^{-\novt} E_+(s_+-\log x,x,y,z')|_{\{x=0\}}, \\
\mce_-(s_-,y,z')= x^{-\novt} E_-(s_-+\log x,x,y,z')|_{\{x=0\}}.
\end{gathered}\label{defepm}
\end{gather}
So we deduce from \eqref{sol-CP} and \eqref{radf1}  that
\begin{gather}
\begin{gathered}
\mcr_+(\vphi,\psi)(s_+,y)=- \p_s  \int_X \left( \mce_+(s_+,y,z')\psi(z')+ \ps \mce_+(s_+,y,z')\vphi(z')\right) \dvol_{g(z')},\\
\mcr_-(\vphi,\psi)(s_-,y)=  \ps  \int_X \left( \mce_-(s_-,y,z')\ \psi(z') +\ps \mce_-(s_-,y,z')\vphi(z')\right) \dvol_{g(z')}.
\end{gathered}\label{newker1}
\end{gather}
If $u(t,z)$ satisfies \eqref{CPWE} with initial data $(\vphi,\psi),$ and  if we set  $V_+(x,s_+,y)=x^{-\novt} u_+(s_+-\log x,z),$  $V_-(x,s_-,y)= x^{-\novt} u_-(s_-+\log x,z),$  and take Fourier transform in $s_\pm,$ then for $x>0$ we obtain
\begin{gather*}
\widehat{V_+}(x,\la,y)= x^{-\novt-i\la}\widehat{u_+}(\la,z)=  x^{-\novt-i\la} R_+(\la)(\psi+i\la\vphi), \;\ x>0, \;\  \im \la<<0, \\
\widehat{V_-}(x,\la,y)= - x^{-\novt+i\la} \widehat{u_-}(\la,z)= x^{-\novt+i\la}R_-(\la)(\psi+i\la\vphi), \;\ x>0, \;\  \im \la>>0.
\end{gather*}

It turns out that the restriction of the Schwartz kernels
\begin{gather}
\begin{gathered}
\left.x^{-\novt-i\la} R_+(\la,x,y,z')\right|_{\{x=0\}}= \mcp_+(\la, y,z'),\;\ z' \in \intx, \;\ \im \la<<0,\\
\left.x^{-\novt+i\la} R_-(\la,x,y,z')\right|_{\{x=0\}}= \mcp_-(\la, y,z'),  \;\ z' \in \intx, \;\  \im \la>>0,
\end{gathered} \label{eisens}
\end{gather}
are  well defined, and are the Schwartz kernels of the adjoint of the forward and backward Poisson operators studied in \cite{GrZw,JS1,M}.  In fact, $\mcp_\pm(\la, y,z')$ also meromorphically continue from  the set $\{\la \in \mc: \pm \im \la <0\}$ to $\mc\setminus \{\pm i(k+\ha), \; k\in \mn\}.$   However, even though  $\widehat{V_\pm}(x,\la,y)|_{\{x=0\}}$ are well defined, it is not clear they are equal to the Fourier transform of the radiation fields $\mcr_\pm(s,y,z')$ in the variable $s,$ and we need the following 
\begin{lemma} \label{conv-FT}  Let $u(t,z)$ be the solution of \eqref{CPWE} with initial data $(\vphi,\psi)\in E_{ac}(X),$ $\vphi, \psi \in C_0^\infty(\intx).$ Let  $V_\pm(x,s,y)$ be defined as above, then for every $f(y) \in C^\infty(\p X),$  and for $\pm \im \la< 0,$
\begin{gather}
\lim_{x\rightarrow 0} \int_\mr \int_{ \p X} e^{-i\la s} \p_s V_\pm(x,s,y) f(y)  \dvol_{h(x)} ds = \int_\mr \int_{ \p X} e^{-i\la s} \p_s V_\pm(0,s,y)f(y) \dvol_{h_0} ds. \label{FTRF}
\end{gather}
\end{lemma}
\begin{proof}   We prove the result in the case of $V_+.$  As usual, multiplying \eqref{CPWE} by $\p_t u$ and integrating by parts, we find that the wave equation has a conserved energy given by
\begin{gather}
\begin{gathered}
E(u(t), \p_t u(t))= || \p_t u(t)||_{L^2(X)}^2 + ||\nabla_g u(t) ||_{L^2(X)}^2- \nsq ||u(t)||_{L^2(X)}^2= \\  || \psi ||_{L^2(X)}^2 + ||\nabla_g \vphi ||_{L^2(X)}^2- \nsq ||\vphi ||_{L^2(X)}^2,
\end{gathered}\label{Est-En}
\end{gather}
provided the initial data is orthogonal to the eigenfunctions of $\Delta_g-\nsq.$ Otherwise, one would have exponentially growing solutions of the form 
$u(t,z)= e^{\pm \mu_j t} \psi_j(z),$ where
$(\Delta_g-\nsq+\mu_j^2)\psi_j=0.$ Moreover,  this energy is positive if and only if $\vphi$ is orthogonal to the eigenfunctions of $\Delta_g.$

We know  from \eqref{resolv0} that for $C>0$ large, 
\begin{gather}
\begin{gathered}
||\psi ||_{H_0^2(X)} \leq  ||\Delta_g \psi||_{L^2(X)}+ C ||\psi||_{L^2(X)}, \;\ \psi\in C_0^\infty(X).
\end{gathered}\label{h2bd}
\end{gather}

Throughout the proof we will use $C$ to denote a constant which may change from line to line. Since $\Delta_g$ commutes with the wave operator, if $u(t,z)$ satisfies \eqref{CPWE} with initial data $(\vphi,\psi)\in E_{ac}(X),$ $\vphi, \psi \in C_0^\infty(\intx),$ we deduce from  \eqref{Est-En} and \eqref{h2bd} that
\begin{gather*}
||\p_t u(t,z)||_{H^1_0(X)}^2 \leq ||\p_t u(t,z)||_{H^2_0(X)}^2\leq C( ||  \Delta_g  \p_t u(t,z)||_{L^2(X)}^2+ ||\p_t u||_{L^2(X)}^2) \leq \\ C (E(\Delta_g\vphi, \Delta_g \psi)+ E(\vphi, \psi)).
\end{gather*}
In particular, if $U\sim [0,\eps) \times \p X$ is a collar neighborhood of $\p X$ where \eqref{prod} holds and  for any $\mu\leq \eps,$
\begin{gather*}
\int_0^\mu \int_{\p X} |x \p_x \p_t u(t,x,y)|^2  \dvol_{h_0} \frac{dx}{x^{n+1}} \leq C \int_0^\eps \int_{\p X} |x \p_x \p_t u(t,x,y)|^2 \dvol_{h(x)} \frac{dx}{x^{n+1}}\leq \\
||\p_t u||_{H_0^1(X)}^2\leq C (E(\Delta_g\vphi, \Delta_g \psi) + E(\vphi,  \psi)).
\end{gather*}
If we set $u(t,z)= x^{\novt} w(t,z),$ then $x\p_x \p_t u=x^{\novt}( x\p_x \p_t w+ \novt \p_t w),$ and since for each $t>0, $ $w(t,x,y) \in C_0^\infty(\intx),$
\begin{gather*}
\int_0^\mu \int_{\p X} |x \p_x \p_t u(t,x,y)|^2 \dvol_{h_0} \frac{dx}{x^{n+1}}= \\ \int_0^\mu \int_{\p X}\left( (x \p_x \p_t w(t,x,y))^2+ \nsq (\p_t w(t,x,y))^2 + \novt x \p_x( \p_t w)^2(t,x,y)) \right) \dvol_{h_0} \frac{dx}{x} =\\
 \int_0^\mu \int_{\p X}\left( (x \p_x \p_t w(t,x,y))^2+ \nsq (\p_t w(t,x,y))^2 \right) \dvol_{h_0}\frac{dx}{x} +\novt \int_{\p X} (\p_t w)^2(t,\mu,y) \dvol_{h_0}.
 \end{gather*}
 In particular this shows that
 \begin{gather*}
 \int_{\p X} (\p_t w)^2(t,\mu,y) \dvol_{h_0} \leq C (E(\Delta_g\vphi, \Delta_g \psi)+ E(\vphi, \psi)), \;\  t>0, \;\  0<\mu \leq \eps.
 \end{gather*}
 In particular, if we restrict this inequality to the curve $t+\log \mu=s,$ $\mu \leq \eps,$ and  since by definition  $V_+(\mu ,s,y)= w(s-\log \mu, \mu, y),$  we have
 \begin{gather*}
 \int_{\p X}  |\p_s V_+(\mu ,s,y)|^2  \dvol_{h_0} \leq C( E(\Delta_g\vphi, \Delta_g \psi)+ E(\vphi, \psi))) , \;\ 0<\mu \leq \eps.
 \end{gather*}
Now we  use  the Cauchy-Schwartz inequality and the fact that $\p X$ is compact to conclude that
\begin{gather*}
 \int_{\p X}  |\p_s V_+(\mu,s,y)|  \dvol_{h(\mu)} \leq  C \int_{\p X}  |\p_s V_+(\mu,s,y)|  \dvol_{h_0} \leq \\ 
 C \left( \int_{\p X}  |\p_s V_+(\mu,s,y)|^2  \dvol_{h_0}\right)^\ha \leq C (E(\Delta_g\vphi, \Delta_g \psi)^\ha+ E(\vphi, \psi)^\ha)), \;\ 0< \mu \leq \eps.
 \end{gather*}
 We know that by finite speed of propagation there exists $s_0\in \mr$ so that $V_+(\mu,s,y)=0$ for $s\leq s_0,$ therefore \eqref{FTRF}  follows by the dominated convergence theorem and the fact that $\im\la<0.$
\end{proof}

It follows from Lemma \ref{conv-FT}  that if $u(t,z)$ satisfies \eqref{CPWE} with initial data $(\vphi,\psi),$ then

\begin{gather*}
\left.\widehat{\p_s V_+}(x,\la,y)\right|_{\{x=0\}}= \widehat{\mcr_+(\vphi,\psi)}(\la,y)= i\la \mcp_+(\la)( \psi+ i\la \vphi) , \; \im \la <0, \\
\left.\widehat{\p_s V_-}(x,\la,y)\right|_{\{x=0\}}=- \widehat{\mcr_-(\vphi,\psi)}(\la,y)=- i\la \mcp_-(\la)( \psi+ i\la \vphi), \; \im \la >0. 
\end{gather*}
But since $\p_s V_\pm(0,s,y) \in L^2(\mr \times \ \p X, ds \dvol_{h_0}),$  it follows from \eqref{FTRF} and the fact that $\mcp_\pm$ continues meromorphically to $\im \la=0,$ that
\begin{gather}
\left.\widehat{\p_s V_\pm }(x,\la,y)\right|_{\{x=0\}}= \pm \widehat{\mcr_\pm (\vphi,\psi)}(\la,y)= \pm i \la \mcp_\pm (\la)( \psi+i\la\vphi), \;  \la\in \mr. \label{FTRF1}
\end{gather}
Since  $\la \widehat{V_\pm}(0,\la,y) \in L^2(\mr \times \ \p X, d\la \dvol_{h_0}),$ this also shows that $\mcp_\pm$ has at most a simple pole at $\la=0.$

 It was shown in \cite{sbhrf}, using arguments of the proof of Proposition 2.2 of \cite{Guillope}, that the maps
\begin{gather*}
\begin{gathered}
\mcp_\pm(\la) : C_0^\infty(\intx) \longrightarrow C^\infty(\mr_+ \times \p X) \\
\psi \longmapsto \sqrt{\frac{2}{\pi}} \int_X \mcp_\pm (\la,y,z') \psi(z') \dvol_{g(z')}, \; \la>0,
\end{gathered}
\end{gather*}
extend to  maps
\begin{gather}
\mcp_\pm (\la) : \Lac \longrightarrow L^2(\mr_+\times \p X; \la^2  d\la \dvol_{h_0}), \label{mapprop}
\end{gather}
and if $f:[0,\infty) \longrightarrow \mc,$
\begin{gather}
\mcp_\pm(\la) f\left(\Delta_g-\nsq\right) \psi= f(\la) \mcp_\pm(\la) \psi, \;\ \psi\in \Lac.  \label{spec-res0}
\end{gather}

In particular,  if we denote $A=\sqrt{\Delta_g-\nsq},$ we have
\begin{gather}
\begin{gathered}
\la \mcp_\pm(\la) \psi= \mcp_\pm A \psi, \;\ \la>0, \; \;\  \psi\in \Lac, \\
\la \mcp_\pm(\la) \psi=- \mcp_\pm A \psi, \;\ \la<0, \; \psi\in \Lac.
\end{gathered}\label{absvP}
\end{gather}
It follows from \eqref{newker1} that
\begin{gather}
\begin{gathered}
\mcp_+(\la,y,z')= \widehat{\mce_+}(\la,y,z'), \\
\mcp_-(\la,y,z')= -\widehat{\mce_-}(\la,y,z').
\end{gathered} \label{Poi-radf}
\end{gather}

For $\la>0,$ the scattering matrix is  defined to be the operator 
\begin{gather}
\begin{gathered}
\mca(\la): L^2(\mr_+ \times \p X, \la^2 d\la \dvol_{h_0}) \longrightarrow L^2(\mr_+ \times \p X, \la^2 d\la \dvol_{h_0})\\
\mca(\la) \circ \mcp_-(\la) = \mcp_+(\la).
\end{gathered}\label{defscatmat}
\end{gather}
Since $\sqrt{\frac{2}{\pi}}\mcp_\pm$ are surjective isometries, 
\begin{gather*}
\mca(\la)= \mcp_+(\la) \circ \mcp_-(\la)^{-1}= \mcp_+(\la) \circ \mcp_-(\la)^*, \;\ \la>0.
\end{gather*}

If we switch $\la$ to $-\la,$ we have
\begin{gather*}
\mcp_+(-\la)=\overline{ \mcp_+(\la)}=\mcp_-(\la) \text{ and } \mcp_-(-\la)=\overline{ \mcp_-(\la)}= \mcp_+(\la),
\end{gather*}
and so
\begin{gather*}
\mca(-\la)= \mca(\la)^{-1}, \; \la>0,
\end{gather*}
and we still have
\begin{gather*}
\mca(-\la)\circ \mcp_-(-\la)= \mcp_+(-\la), \;\ \la>0,
\end{gather*}
and therefore
\begin{gather}
\begin{gathered}
\mca(\la) \mcp_-(\la)= \mcp_+(\la), \;\ \la>0 \text{ and  }  \mca(\la) \mcp_-(\la)= \mcp_+(\la), \;\ \la<0, \\
\text{ and in particular } \\
\mca(\la) = (\one_+(\la) \mcp_+(\la))\circ(\one_+(\la) \mcp_-(\la))^{-1}+ (\one_-(\la) \mcp_+(\la))\circ(\one_-(\la) \mcp_-(\la))^{-1}
\end{gathered}\label{scatmat0902}.
\end{gather}

 In view of \eqref{sol-CP},  the scattering operator is the map
\begin{gather}
\mcs F(s,y) = -\mcf^{-1} \mca(\la) \mcf F, \label{def-scat-mat0}
\end{gather}
where $\mcf$ denotes the Fourier transform in $s.$  But  since $\sqrt{\frac{2}{\pi}}\one_\pm(\la)\la\mcp_\pm(\la)$ is a unitary operator, before we translate the right hand side of this equation into the $s$ variable, we need to consider \eqref{mapprop}, and rewrite
\begin{gather}
\mca(\la)= (\one_+(\la) \la \mcp_+(\la))\circ(\one_+(\la)  \la\mcp_-(\la))^{*}+ (\one_-(\la) \la \mcp_+(\la))\circ(\one_-(\la) \la \mcp_-(\la))^{*}, \label{defn0904}
\end{gather}
 and therefore \eqref{Poi-radf} implies that 
\begin{gather}
\begin{gathered}
\mcs=- (\one_+(D_s) \p_s \mce_+)\circ(\one_+(D_s) \p_s \mce_-)^{*}-  (\one_-(D_s) \p_s \mce_+)\circ(\one_-(D_s) \p_s \mce_-)^{*},
\end{gathered}\label{split-scat}
\end{gather}
where $\one_\pm(D_s)$ is the operator defined by
\begin{gather*}
\one_\pm(D_s) f(s)=\frac{1}{2\pi}\int_{\mr} e^{i\sigma s} \one_\pm(\sigma) \widehat{f}(\sigma) \; d\sigma.
\end{gather*}

So we only need to analyze each one of the  terms  of \eqref{split-scat} separately. Perhaps it is worth explaining this fact  in terms of propagation of singularities.  The characteristic variety of the wave operator splits in two connected components corresponding to the sign of $\tau=\sigma,$ which are the variables dual to $t$ and $s$ respectively.  Since the coefficients of the wave operator do not depend on $t,$ $\tau$ remains constant along the bicharacteristics, and in particular the sign of $\tau$ remains constant.  Therefore equation \eqref{split-scat} just says that the scattering operator splits into two parts corresponding to the sign of $\tau.$  This will be discussed more thoroughly in the next section.

 The fundamental solution of the Cauchy problem \eqref{CPWE} with data $(0,\psi),$ is the odd part of the wave group $U(t)$ which is given by 
 $A^{-1}\sin( t A),$  where  $\sin (tA)= \frac{1}{2i}( e^{i t A}- e^{-it A}).$     It follows from \eqref{sol-CP} that if $\mcz(t,z,z')$ denotes the kernel of $A^{-1}\sin(tA),$ then
 \begin{gather*}
 E_+(t,z,z')= \one_+(t) \mcz(t,z,z').
 \end{gather*}
If $W(t,z,z')$ is the Schwartz kernel of $e^{itA},$ then
\begin{gather*}
\frac{1}{2}(W(t,z,z')+ W(-t,z,z'))= \cos(t A)= \p_t \mcz(t,z,z'),
\end{gather*}
 and we will compute 
 \begin{gather*}
 \lim_{x\rightarrow 0} \lim_{x'\rightarrow 0} (\one_+W)(s-s'-\log x-\log x') \text{ and } 
  \lim_{x\rightarrow 0} \lim_{x'\rightarrow 0} (\one_+W)(-s+s'+\log x+\log x'),
  \end{gather*}
$\one_+=\one_+(t),$   in the sense of \eqref{sep-limits}.

 For $\psi\in C_0^\infty(\intx)\cap \Lac,$  $u(t,z)=e^{ i t A}\psi$  is the solution to
 \begin{gather}
 \begin{gathered}
 (D_t- A) u (t,z)=0 ,\\
 u (0,z)=\psi(z), \;\ \psi\in \Lac
 \end{gathered}\label{modwg}
 \end{gather}
and in particular,  it also satisfies
\begin{gather}
\begin{gathered}
(D_t^2-\Delta_g+\nsq) u(t,z)=0, \\
u(0,z)=\psi(z), \;\ \p_t u(0,z)= i A \psi.
\end{gathered} \label{CPart}
\end{gather}

Then in view of \eqref{sol-CP},
\begin{gather}
\begin{gathered}
\lim _{x\rightarrow 0} x^{-\novt} (\one_+u)(s-\log x,x,y)= -\mce_+ \left(i A \psi\right)(s,y)- \p_s \mce_+ \psi(s,y), \\
\lim _{x\rightarrow 0} x^{-\novt} (\one_-u)(s+\log x,x,y)= \mce_- \left(i A \psi\right)(s,y)+ \p_s \mce_+ \psi(s,y), \\
\end{gathered} \label{help-eq0904}
\end{gather}
But we know from \eqref{absvP} that
\begin{gather*}
\widehat{\ps \mce_\pm} (\psi)= i\la \mcp_\pm(\la)(\psi)= i\one_+(\la) \mcp_\pm (iA\psi) - \one_-(\la) \mcp_\pm(\la) (i A\psi),
\end{gather*}
and therefore we conclude that
\begin{gather}
\begin{gathered}
\lim _{x\rightarrow 0} x^{-\novt}  (\one_+u)(s-\log x,x,y)= -2\one_+(D_s) \p_s \mce_+(\psi)(s,y), \\
\lim _{x\rightarrow 0} x^{-\novt} (\one_-u)(s+\log x,x,y)=2 \one_+(D_s) \p_s \mce_-(\psi)(s,y).
\end{gathered}\label{limu}
\end{gather}

If we now work with the group $W_-(t)=e^{-it A},$ then $v(t,z)=e^{-itA} \psi(z)$ satisfies
\begin{gather*}
(D_t+A) v(t,z)=0, \\
v(0,z)=\psi(z)\in \Lac.
\end{gather*}
If we proceed as above, and now use $-A$ in place of $A$ in \eqref{help-eq0904},   we  obtain
\begin{gather}
\begin{gathered}
\lim _{x\rightarrow 0} x^{-\novt}  (\one_+)v(s-\log x,x,y)= -2\one_-(D_s) \p_s \mce_+(\psi)(s,y), \\
\lim _{x\rightarrow 0} x^{-\novt} (\one_-v)(s+\log x,x,y)=2 \one_-(D_s) \p_s \mce_-(\psi)(s,y).
\end{gathered}\label{limu3}
\end{gather}

We will now compute  $\lim_{x'\rightarrow 0} {x'}^{-\novt} (\one_+W)(t-s'-\log x', z,z'),$ $t>0.$  Notice that since $t$ and $s'$ remain fixed, the factor $\one_+$ becomes irrelevant.  Since $W(t,z,z')$ is the Schwartz kernel of  $e^{it A},$    we know by the group property that
\begin{gather*}
W(t+t',z,z')= \int_{X} W (t,z,w) W (t',w,z') \dvol_{g(w)}.
\end{gather*}
Then for  $t>0$ and $t'=-s'-\log x'>0,$ we have 
\begin{gather*}
{x'}^{-\novt} (\one_+W)(t-s'-\log x',z,z')= {x'}^{-\novt} W(t-s'-\log x',z,z')= \\ \int_X W(t,z,w) {x'}^{-\novt} W (-s'-\log x',w,z') \dvol_{g(w)}=
\int_X W(t,z,w) {x'}^{-\novt} \overline{W (s'+\log x',w,z')} \dvol_{g(w)},
\end{gather*}
where we used that $\overline{W(t,w,z)}= W(-t,w,z).$

We deduce from \eqref{limu} that for $\psi\in C_0^\infty(\intx) \cap \Lac,$ the limit as $x'\searrow 0$  is equal to
\begin{gather*}
\lim_{x'\rightarrow 0} \int_X   \psi(z)  {x'}^{-\novt} W(t-s'-\log x',z,z') \dvol_{g(z)}= \\ \lim_{x'\rightarrow 0} \int_X  \left(\int_X \psi(z)  W(t,z,w)   \dvol_{g(z)}\right) {x'}^{-\novt} \overline{W(s'+\log x',w,z')} \dvol_{g(w)}= \\
 2  \int_X \left( \int_{X} \psi(z)  W(t,z, w) \dvol_{g(z)}\right)  \overline{ \one_+(D_{s'})\p_{s'}\mce_-(s',w,y') } \dvol_{g(w)}. 
 \end{gather*}
If $\psi\in C_0^\infty,$   $ \int_{X} \psi(z) W(t,z, w) \dvol_{g(z)} \in H_0^k(X) \cap \Lac,$ for all $k$ and  it was shown in \cite{HoSa} using the energy estimates established in \cite{sbhrf}  that if $(f_1,f_2)\in E_{ac}(X)$ and if $V_\pm(x,s,y)$ is defined as in \eqref{radf}, then \eqref{radf1} still makes sense. Even though $V_\pm$ is not $C^\infty,$ but the restrictions of \eqref{radf1} are well defined, see the discussion between equation (3.15) and (3.18) of \cite{HoSa}. 
So we conclude that in the sense of distributions
\begin{gather}
\begin{gathered}
\lim_{x'\rightarrow 0}  {x'}^{-\novt} W(t-s'-\log x',z,z') \; \dot{ = }\;  \mcm_1(t,z,s',y')= \\ 2 \int_X  W(t,z, w)\overline{ \one_+(D_{s'})  \p_{s'}\mce_-(s',w,y')} \dvol_{g(w)}. 
\end{gathered}\label{sep-lim0}
\end{gather}

Now we can take the second limit, and we  pick $F(s,y)$ such that $\p_s^k F\in L^2(\mr \times \p X).$  We know from \eqref{mapprop} that
\begin{gather*}
\one_+(D_s)( \p_s \mce_-)^*F(w)= \int_{\mr\times \p X}  \overline{\one_+(D_{s'}) \p_{s'}\mce_-(s',w,y')} F(s',y') ds' \dvol_{h_0(y')}= f(w) \in \Lac,
\end{gather*}
and hence we deduce from \eqref{limu},  and  again the fact that the limit \eqref{radf1} still makes sense for data in $E_{\operatorname{ac}}(X),$ that
\begin{gather}
\begin{gathered}
\lim_{x\rightarrow 0} x^{-\novt} \int_{\mr \times \p X} \mcm_1(s-\log x, x,y, s',y') F(s',y') ds'\dvol_{h_0(y')}= \\
2 \lim_{x\rightarrow 0}  x^{-\novt} \int_{X} W(s-\log x,x,y, w) f(w) \dvol_{g(w)}=  \\
-4\one_+(D_s)  \int_{X}  \p_s  \mce_+(s,y, w) \left( \int_{\mr\times \p X}  \overline{\one_+(D_{s'}) \p_{s'}\mce_-(s',w,y')} F(s',y') ds' \dvol_{h_0(y')}\right) \dvol_{g(w)}.
\end{gathered}\label{sep-lim1}
\end{gather}

So we conclude that, in the sense of distributions
\begin{gather*}
\begin{gathered}
\lim_{x\rightarrow 0} x^{-\novt}  \mcm_1(s-\log x, x,y, s',y')= \\
-4 \int_X \one_+(D_s) \p_s \mce_+(s,y,w) \overline{ \one_+(D_{s'}) \p_{s'} \mce_-(s',w,y')}  \dvol_{g(w)}. 
\end{gathered}
\end{gather*}

 So finally we arrive at
\begin{gather}
\begin{gathered}
\lim_{x\rightarrow 0} \lim_{x'\rightarrow 0} (xx')^{-\novt} W(s-s'-\log x-\log x',z,z')= \\-4  \int_X \one_+(D_s) \p_s \mce_+(s,y,w)\overline{\one_+(D_{s'}) \p_{s'}\mce_-(s',w,y')} \dvol_{g(w)}\; \dot{=}\;
- 4\mck_{\mcs_+}(s,y,y',s'), \\
\text{ where }  \mck_{\mcs_+} \text{ is the Schwartz kernel of } \mcs_+= (\one_+(D_s) \mce_+ )\circ (\one_+(D_s) \mce_-)^*.
\end{gathered}\label{plus-scat}
\end{gather}

Perhaps one should also explain this in terms of propagation of singularities.  The characteristic variety of the operator $D_t-A$ is given by $\{\tau=\sigma(A)\},$ where $\sigma(A)$ is the principal symbol of $A.$ Hence $\tau>0,$ and the corresponding scattering operator is restricted to the $\{\tau>0\}$  component of the characteristic variety of $D_t^2-A^2.$

Next we need to do the same  computations for  $t= -s+s'+\log x +\log x'.$    But we now work with the group $W_-(t)= e^{-it A}=W(-t).$
Again, by the group property  we  have
\begin{gather*}
(xx')^{-\novt} W(-s+s'+\log x+\log x',z,z')= (xx')^{-\novt} W_-(s-s'-\log x-\log x')= \\ \int_X x^{-\novt} W_-(s-\log x,z,w) {x'}^{-\novt} W_-(-s'-\log x',w,z') \dvol_{g(w)}=\\
 \int_X x^{-\novt} W_-(s-\log x,z,w) {x'}^{-\novt} \overline{W_-(s'+\log x',w,z')} \dvol_{g(w)}=\\
\end{gather*}
We then proceed exactly as in the previous case, and now use \eqref{limu3} instead of \eqref{limu}, and if we repeat the same arguments used above, we  find that
\begin{gather}
\begin{gathered}
\lim_{x\rightarrow 0} \lim_{x'\rightarrow 0} (xx')^{-\novt}  W(-s+s'-\log x-\log x',z,z')=\\ -4\int_X \one_-(D_s) \p_{s} \mce_+(s,y,w)\overline{\one_-(D_{s'} \p_{s'}\mce_-(s',w,y')} \dvol_{g(w)}\;  \dot{=} \; - 4\mck_-(s,y,y',s'), \\
\text{ where }  \mck_{\mcs_-} \text{ is the Schwartz kernel of } \mcs_-= (\one_-(D_s) \mce_+ )\circ (\one_-(D_s) \mce_-)^*.
\end{gathered}\label{plus-scat1}
\end{gather}
Here of course,  $\tau<0$ on the characteristic variety of $D_t+A,$ and hence the corresponding scattering operator is restricted to the $\{\tau<0\}$  component of the characteristic variety of the wave operator.

So we finally conclude  from \eqref{split-scat}  that
\begin{gather*}
\mck_{\mcs} (s,y,s',y')= -\mck_{\mcs_+}(s,y,y',s')- \mck_{\mcs_-}(s,y,y,,s')=\\ \oq \lim_{x\rightarrow 0} \lim_{x'\rightarrow 0} (xx')^{-\novt} \left( e^{i (s-s'-\log x-\log x') A}+ e^{i(-s+s'+\log x+\log x') A}\right)= \\
\ha   \lim_{x\rightarrow 0} \lim_{x'\rightarrow 0} (xx')^{-\novt}  \cos \left((s-s'-\log x -\log x') A\right)= \\ \ha   \lim_{x\rightarrow 0} \lim_{x'\rightarrow 0} (xx')^{-\novt} \p_s A^{-1} \sin \left((s-s'-\log x -\log x') A\right)= \\ \ha   \lim_{x\rightarrow 0} \lim_{x'\rightarrow 0} (xx')^{-\novt} \p_s E_+(s-s'-\log x-\log x')
\end{gather*}
and this proves \eqref{kerscat-0}.

\section{The  underlying Lagrangian submanifolds}\label{ULM}

   The microlocal structure of the Schwartz kernel of  $\mcw(t)= A^{-1}\sin(tA),$  $A=\sqrt{\Delta_g-\nsq},$  as a distribution in $\mr \times \intx \times \intx$ is well known due the work of H\"ormander \cite{Hormander-FIO} when $t>0$ and to Duistermaat and H\"ormander \cite{DuHo} and Melrose and Uhlmann \cite{MU} up to $t=0.$   In particular, if the manifold $(\intx,g)$ is non-trapping, we know that for $t>0,$  $\mcw=\mcw_++\mcw_-$ where the Schwartz kernel of $\mcw_\pm,$ which we denote by $\mck_{\mcw_\pm(t)},$ is a Lagrangian distribution in  
\begin{gather*}
\mck_{\mcw_\pm} \in I^{-\fq}(\mr \times \intx \times \intx, \La_\pm, \Omega_{\mr\times \intx \times \intx}^\ha),
\end{gather*}
 and $\La_\pm$ is defined below in \eqref{deflapm0}, see for example  Theorem 5.1.2 of \cite{Duistermaat}.
 The non-trapping condition is necessary to guarantee that $\La_\pm$ are $C^\infty$ Lagrangian submanifolds, see for example \cite{Duistermaat,Hormander}.
   
    Our goal is to understand the microlocal structure of the limit \eqref{kerscat-0}, and to do that we first investigate the global microlocal structure of $(xx')^{-\novt}\mck_{\mcw_\pm }(s-\log x-\log x', z,z')$  and then investigate their behavior as $x, x'\searrow 0.$  We will work in  
  $T^*(\mr_t \times \intx\times\intx)\setminus 0$ equipped with  the canonical $2$-form 
\beq 
\widetilde\omega = d\tau \wedge dt+ \sum_{j = 1}^{n+1} d\zeta_j\wedge dz_j + \sum_{j = 1}^{n+1} d\zeta'_j\wedge dz'_j,
\eeq
where $(z,\zeta)$ and $(z',\zeta')$  will denote coordinates on the left and right factors of $T^*\intx\times T^*\intx$  respectively.

We will work with $-\ha\square=\ha(-D_t^2+\Delta_g),$  and we also distinguish between the lifts of the wave operator to the right or left factors of $X\times X.$  The principal symbol on the right and left  factors of $T^* \mr \times T^*\intx\times T^*\intx$  are defined to be  respectively 
 \begin{gather}
 Q_R(\tau, z,\zeta,z',\zeta')=\ha(-\tau^2+|\zeta'|_{g^*(z')}^2) \text{  and   } Q_L(\tau,z,\zeta,z',\zeta')=\ha(-\tau^2+|\zeta|_{g^*(z)}^2), \label{eqqr}
 \end{gather}
where $g^*$ is the dual metric to $g,$ and we will think of these as functions on  $T^*(\mbr\times\intx\times \intx).$ Their characteristic varieties are 
$\mcn_{Q_\bullet}=\{ (t,\tau, z,\zeta,z',\zeta'): Q_\bullet=0\}.$

 In local coordinates, the Hamilton vector field of $Q_\bullet,$ $\bullet=R,L,$ with respect to the canonical form $\widetilde \omega$  is given by
\begin{gather}
\begin{gathered}
H_{Q_L}= -\tau\frac{\p}{\p t} + \ha H_{|\zeta|^2}, \text{ where }   |\zeta|^2=|\zeta|_{g^*(z)}^2 \\ 
  \text{and } H_{|\zeta|^2}= \sum_{j=1}^{n+1}( \frac{\p |\zeta|^2}{\p {\zeta_j}}\frac{\p}{\p{z_j}}-\frac{\p |\zeta|^2}{ \p{z_j}}\frac{\p}{ \p{\zeta_j}}), \\
H_{Q_R}= -\tau\frac{\p}{\p t} + \ha H_{|\zeta'|^2}, \text{ where }   |\zeta'|^2=|\zeta'|_{g^*(z')}^2 \\ 
  \text{and } H_{|\zeta'|^2}= \sum_{j=1}^{n+1}( \frac{\p |\zeta'|^2}{\p {\zeta'_j}}\frac{\p}{\p{z'_j}}-\frac{\p |\zeta'|^2}{ \p{z'_j}}\frac{\p}{ \p{\zeta'_j}}).
\end{gathered}\label{HQ}
\end{gather}

We define the  bicharacteristic relation for $Q_L$
 \begin{gather}
\begin{gathered}
\Lambda=   \{(t,\tau,z,\zeta, z',-\zeta')\in T^*(\mr \times \intx \times \intx)\setminus 0: \; (t,\tau,z,\zeta)\in \mcn_{Q_L}, \;\
  (0,\tau,z',-\zeta')\in \mcn_{Q_L}, \\ \text{ lie on the same integral curve of } H_{Q_L}  \}.
  \end{gathered}\label{fulllag}
  \end{gather}
This definition is unchanged if we use $Q_R$ instead of $Q_L,$ since this just switches the roles of $(z,\zeta)$ and $(z',\zeta').$    We can also define $\Lambda$ as the flow-out of 
\begin{gather}
 \La^0= \{(t,\tau,z,\zeta, z', \zeta') \in T^*(\mr\times \intx \times \intx): \; t=0,\;  z=z',\;  \zeta=-\zeta',\;  \; \tau^2= |\zeta|_{g^*(z)}^2 \} \label{diagonal0}
 \end{gather} 
along the integral curves of $H_{Q_\bullet}.$ In other words,
\begin{gather}
\La= \bigcup_{\gamma \in \mr} \exp( \gamma H_{Q_R}) \La^0=  \bigcup_{\gamma \in \mr} \exp( \gamma H_{Q_L}) \La^0. \label{defla-J}
\end{gather}
Since the vector fields $H_{Q_R}$ and $H_{Q_L}$ commute,
\begin{gather}
\La= \bigcup_{\gamma_1, \gamma_2 \in \mr} \exp( \gamma_1 H_{Q_R}) \circ  \exp( \gamma_2 H_{Q_L}) \La^0=\bigcup_{\gamma_1, \gamma_2 \in \mr} \exp( \gamma_1 H_{Q_L}) \circ  \exp( \gamma_2 H_{Q_R}) \La^0. \label{defla-J1}
\end{gather}

 Notice that $\tau$ is constant along the integral curves of $H_{Q_\bullet},$ $\bullet=R,L,$ and   in view of the non-trapping assumption, $\La$ is a $C^\infty,$ conic, closed Lagrangian submanifold in  
$T^*(\mr \times \intx \times \mr \times \intx)\setminus 0,$ see Theorem 26.1.13 of \cite{Hormander}.   Since $\mcn_{Q_\bullet} \setminus 0$ consists of two disjoint components
\begin{gather*}
\mcn_{Q_\bullet} \setminus 0=\mcn_{Q_\bullet ,+}\cup\mcn_{Q_\bullet,-}, \text{ where } \\
\mcn_{Q_L,+}=\{\tau= -|\zeta|_{g^*{(z)}}\}, \;\ \mcn_{Q_L,-}=\{\tau= |\zeta|_{g^*{(z)}}\}, \\
\mcn_{Q_R,+}=\{\tau= -|\zeta'|_{g^*{(z')}}\}, \;\ \mcn_{Q_R,-}=\{\tau= |\zeta'|_{g^*{(z')}}\}.
\end{gather*}
We shall denote
\begin{gather}
\La=\La_+\cup \La_-, \text{ where } \La_\pm=\La\cap \mcn_{Q_\bullet,\pm}, \;\ \bullet=R,L. \label{deflapm0}
\end{gather}
Of course, in view of \eqref{defla-J}, the definition of $\La_\pm$ is independent of the choice of either $Q_\bullet.$ So the vector fields
\begin{gather}
\begin{gathered}
H_{+R}= |\zeta'| \p_t + \ha H_{|\zeta'|^2}   \text{ and }  H_{+L}= |\zeta| \p_t + \ha H_{|\zeta|^2}   \text{ are  tangent to } \La_+, \\
H_{-R}= - |\zeta'| \p_t + \ha H_{|\zeta'|^2}   \text{ and }  H_{-L}= -|\zeta| \p_t + \ha H_{|\zeta|^2}   \text{ are  tangent to } \La_-.
\end{gathered}\label{HQ2}
\end{gather}
 The vector fields $H_{\pm R}$ and $H_{\pm L}$ obviously commute, and therefore,  for $t_1, t_2\in \mr$ and a point $(t,\tau,z,\zeta,z',\zeta')\in T^*(\mbr\times\intx\times\intx)\setminus 0,$
\begin{gather*}
\exp t_2 H_{\pm L}\circ \exp t_1 H_{\pm R}(t,\tau,z,\zeta,z',\zeta')= \exp t_1 H_{\pm R}\circ \exp t_2 H_{\pm L}(t,\tau,z,\zeta,z',\zeta').
\end{gather*}
  Moreover, away from 
\begin{gather}
 \La_{\pm}^0= \{(t,\tau,z,\zeta, z', \zeta') \in T^*(\mr\times \intx \times \intx): \; t=0,\;  z=z',\;  \zeta=-\zeta',\;  \; \tau= \mp |\zeta|_{g^*(z)} \}, \label{diagonal1}
 \end{gather}

\begin{gather}
\La_\pm \setminus \La_{\pm }^0= \La_{\pm ,R} \cup \La_{\pm ,L}, \label{decompwtla}
\end{gather}
where
\begin{gather}
\begin{gathered}
\La_{\pm ,R} \doteq \bigcup_{\gamma > 0} \exp \gamma H_{Q_{\pm R}} (\La_{\pm }^0) \text{ and }
\La_{\pm ,L} \doteq \bigcup_{\gamma > 0} \exp \gamma H_{Q_{\pm L}} (\La_{\pm }^0).
\end{gathered}\label{eqlat}
\end{gather}
Observe that the relations $\La_{\pm,R}$ and $\La_{\pm,L},$ with the same sign are the inverse to each other.  To see that one just has to realize that  if $(t,\tau,z,\zeta)=\exp(\gamma H_{Q_{\pm R}})(t_1,\tau,z',\zeta'),$ then
$(t_1,\tau,z',\zeta')=\exp(\gamma H_{Q_{\pm L}})(t,\tau,z,\zeta).$   

 But to understand the global behavior of $\La_\pm,$  and the geometric structure of the radiation fields,  we will need to work on the blown-up space $\xo$ defined  above  and to deal with the radiation fields  we define
\begin{gather}
\begin{gathered}
\mcm_f : \mr\times \left(\xo \setminus (R\cup L)\right) \longrightarrow \mr \times \xo \\
(s, m) \longmapsto ( s -  \log \rho_L(m) -  \log \rho_R(m), m)=(t,m), \\
\mcm_b : \mr\times \left(\xo \setminus (R\cup L)\right) \longrightarrow \mr \times \xo \\
(s, m) \longmapsto ( s +  \log \rho_L(m) +  \log \rho_R(m), m)=(t,m),
\end{gathered}\label{sch}
\end{gather}
We define the corresponding forward and backward blow-ups
\begin{gather}
\begin{gathered}
\beta_{1f}=\beta_0\circ \mcm_f : \mr_s \times \left(\xo \setminus(R\cup L)\right) \longrightarrow \mr_t\times X \times X, \\
\beta_{1b}=\beta_0\circ \mcm_b : \mr_s \times \left(\xo \setminus(R\cup L)\right) \longrightarrow \mr_t\times X \times X.
\end{gathered} \label{defbeta1}
\end{gather}

 We will prove the following 
\begin{theorem}\label{soj}
Let $(\intx, g)$ be a non-trapping AHM.  Let $\rho_L, \rho_R$ be  boundary defining functions of $L$ and $R$ respectively.   Let  $\La_\pm \subset T^*(\mbr_t\times \intx\times\intx)$ be the $C^\infty$ Lagrangian submanifolds defined in \eqref{eqlat} and let $\beta_{1f}^*\La_\pm$  and
$\beta_{1b}^*\La_\pm$ denote the lifts of $\La_\pm$ by $\beta_{1f}$ and $\beta_{1b}$ respectively  in the interior of $\xo.$
 Then  $\beta_{1f}^*\La_\pm$ and  $\beta_{1b}^*\La_\pm$ have smooth extensions up to the boundary of  $T^*(\mbr_s \times \xo)$ which intersect the right, left and front faces transversally and are closed in $T^*(\mbr_s \times \xo)\setminus 0.$  Moreover if $\wtla_{\pm}^f$ denotes the extension of $\beta_{1f}^*\La_\pm$ and
 $\wtla_{\pm}^b$ denotes the extension of $\beta_{1b}^*\La_\pm,$ then 
 \begin{gather*}
 \wtla_{\pm}^\bullet \cap \{\rho_R=0\} \text{ is a Lagrangian submanifold of }  T^*(\mr_s \times X \times_0 \p X),\; \bullet= f,b, \\
 \wtla_{\pm}^\bullet \cap \{\rho_L=0\} \text{ is a Lagrangian submanifold of }  T^*(\mr_s \times \p X \times_0 X), \; \bullet= f,b, \\
  \wtla_{\pm}^\bullet \cap \{\rho_R=\rho_L=0\} \text{ is a Lagrangian submanifold of }  T^*(\mr_s \times \p X \times_0 \p X), \; \bullet= f,b.
 \end{gather*}
\end{theorem}

A slightly different version of this theorem was proved in \cite{SaWang}, but  we will prove it for the convenience of the reader.    Similar results associated with the construction of the semiclassical resolvent  were proved by  Melrose, S\'a Barreto and Vasy \cite{MSV}, Chen and Hassell \cite{ChenHa} and by Wang \cite{Wang}.  

Notice that the change of variables $t\longmapsto s=t+\gamma,$ where $\gamma=\log \rho_R+\log \rho_L$ induces a map on $T^*(\mr \times \xo)$ which amounts to the shift along the fibers of $T^*(\xo)$ by $d\gamma.$ Namely,
\begin{gather}
\begin{gathered}
S: T^*(\mr \times \xo) \longrightarrow T^*(\mr \times \xo) \\
(m,\tau,\nu)   \longmapsto (m, \sigma, \nu+ d\gamma).
\end{gathered}\label{shift-map}
\end{gather}

One can reinterpret Theorem \ref{soj} as
\begin{gather}
\begin{gathered}
\beta_{1f}^*\La_+= \beta_0^* \La_+ + \sigma d\gamma, \\
\beta_{1b}^*\La_+= \beta_0^* \La_+ - \sigma d\gamma. 
\end{gathered}\label{shift}
\end{gather}
The analogue of \eqref{shift}  in the semiclassical case was observed by Chen and Hassell \cite{ChenHa} and by Wang \cite{Wang} and by S\'a Barreto and Wang \cite{SaWang}

 The key to proving this result, see Proposition \ref{vfields} below, is that  if  $q_{\bullet R}=-\frac{1}{\rho_R}\beta_{1\bullet}^*Q_R$ and
$q_{\bullet L}= -\frac{1}{\rho_L} \beta_{1\bullet}^* Q_L,$ $\bullet=b,f,$ then $q_{\bullet R}$ and $q_{\bullet L}$ extend to  functions in $C^\infty(T^*(\mr_s\times \xo))$ and the Hamilton vector fields  $H_{q_{\bullet L}}$ and $H_{q_{\bullet R}}$ are tangent to $\mr_s\times \ff$ and away from  $\sigma=0$ (where $\sigma$ is the dual variable to $s$) $H_{q_{\bullet R}}$ is transversal to $\mr_s\times R,$ and  tangent to $\mr_s \times L,$ while $H_{q_{\bullet L}}$ is transversal to $\mr_s\times L$ and tangent to $\mr_s \times R.$ 
\begin{proof}
  We will work with $\beta_{1f},$  but the case of $\beta_{1b}$ is identical.  First, notice that the result is independent of the choice of $\rho_R$ or $\rho_L.$   If  $\tilde\rho_L, \tilde \rho_R$ are  boundary defining functions of the left and right faces, then $\rho_L = \tilde \rho_L f_L$ and $\rho_R = \tilde \rho_R f_R$ for some $f_L, f_R \in C^\infty(\xo)$ with $f_L>0,$ $f_R > 0.$ If  $\tilde{s}= t + \log \tilde \rho_L + \log \tilde \rho_R$ and $s=t + \log\rho_R+\log\rho_L,$ then $\tilde{s}=  s + \log(f_L f_R),$ and the map $(s,m)\mapsto(\tilde{s},m)$ is a global diffeomorphism of $\mr_s\times \xo.$

 As mentioned above, the main ingredient in the proof of Theorem \ref{soj} is the following 
 \begin{prop}\label{vfields}  Let $\rho_R, \rho_L\in C^\infty(\xo)$ be defining functions of  $R$ and $L$ respectively.   Let $\beta_{1f}$ and $\beta_{1b}$ be the maps defined in \eqref{sch}  and let   $q_{\bullet R}=-\frac{1}{\rho_R}\beta_{1\bullet}^*Q_R$ and
$q_{\bullet L}= -\frac{1}{\rho_L} \beta_{1\bullet}^* Q_L,$ $\bullet=b,f.$ Then $q_{\bullet R}$ and $q_{\bullet L}$ extend to  functions in $C^\infty(T^*(\mr_s\times \xo))$ and the Hamilton vector fields  $H_{q_{\bullet L}}$ and $H_{q_{\bullet R}}$ are tangent to $\mr_s\times \ff.$ Moreover, if  $\sigma$ is the dual variable to $s,$ then away from $\sigma=0,$ $H_{q_{\bullet R}}$ is transversal to $\mr_s\times R,$ and $H_{q_{\bullet L}}$ is transversal to $\mr_s\times L.$
 \end{prop}

 The proof of this result is carried out in in a more general setting in Theorem 6.1 and Theorem 6.8 of \cite{SaWang}, but we will do it again here in this particular case, for the convenience of the reader.
 \bpf   We will prove this Proposition in local coordinates valid near $\p(\mr_s \times \xo).$   First, we choose local coordinates  $z = (x, y)$ and $z' = (x', y')$ in which \eqref{prod} holds.  
We  divide the boundary of $\mr_s\times \xo$ into four regions: \\
Region 1:  Near $\mr_s \times L$ and away from $\mr_s\times (R\cup \ff),$ or near $\mr_s\times R$ and away from $\mr_s\times (L \cup \ff).$ \\
Region 2: Near $\mbr_s\times (L \cap \ff)$ and away from $\mbr_s\times R$, or near $\mbr_s\times (R\cap \ff)$ and away from $\mbr_s\times L.$ \\
Region 3:  Near $\mbr_s\times (L\cap R)$ but away from $\mbr_s\times \ff.$ \\
Region 4: Near $\mbr_s\times (L\cap R \cap \ff).$

  First we analyze region 1, near $\mr_s\times L $ but away from $\mr_s \times R$ and $\mr_s \times \ff.$  The case near $\mr_s\times R$ but away from $\mr_s\times L$ and $\mr_s\times \ff$ is identical.   Since we are away from $R,$ we have $\rho_R>\del,$ for some $\del>0,$  and hence $\log \rho_R$ is $C^\infty.$ In this region we may take $x$ as a defining function of $L,$ and instead of \eqref{sch}, we set $s = t + \log x$.   In fact,   the map $(s,m) \longmapsto (s+\log\rho_R,m)$ is a diffeomorphism in the region where $\rho_R>\del,$ and hence  the statements about $q_L$ and $H_{q_L}$ in the lemma are true in this region whether we take $s=t+\log x$ or $s=t+\log x + \log \rho_R.$  In the case near $\mr_s\times R$ but away from $\mr_s\times L$ and $\mr_s\times  \ff$ one sets $s=t+\log x'.$  These particular cases were studied in  \cite{SW}.
  
The change of variables
\begin{gather}
s=t+\log x \label{souj1}
\end{gather}
induces the symplectic change on $T^*(\mbr\times \intx\times \intx)$

\begin{gather}\label{symch}
\begin{gathered}
(x,y,\xi,\eta,t,\tau) \longmapsto (x,y,\txi,\eta,s,\sigma), \\
\text{ where } \txi= \xi-\frac{1}{x}\tau, \ \ \sigma = \tau.
 \end{gathered}
  \end{gather}
  
 In coordinates \eqref{prod}, 
 \begin{gather*}
 g_L^*(x,y,\xi,\eta)= x^2 \xi^2+ x^2 h(x,y,\eta),
 \end{gather*}
 and so
 \begin{gather*}
 \beta_1^*Q_L=-x\sigma\txi-\ha x^2(\txi^2+h(x,y,\eta)), \text{ and hence } \\
   q_L= -\frac{1}{\rho_L} \beta_1^*Q_L=\sigma\txi+\ha x(\txi^2+h(x,y,\eta)).
 \end{gather*}

We have
 \begin{gather*}
 H_{q_L}= (\sigma+x\txi)\p_x+\txi\p_s +\ha x H_{h(x,y,\eta)}- \ha(\txi^2+h(x,y,\eta)+ x\p_x h(x,y,\eta)) \p_{\txi}.
 \end{gather*}
 In particular, $\sigma$ remains constant along the integral curves of $H_{q_R},$ and
 \begin{gather*}
H_{q_L}|_{\{x=0\}}= \sigma \p_x+\txi\p_s - \ha(\txi^2 + h(0, y, \eta))\p_{\txi}.
\end{gather*}
So if $\sigma\not=0,$  $H_{q_L}$ is transversal to $\p X.$

Next we work in region 2 near $\mr_s \times (L\cap \ff),$ but away from $\mr_s\times R$.  The case near $\mr_s\times (R\cap \ff)$ but away from $\mr_s\times L$ is very similar.   In this case, $\rho_R=x'/R>\del,$ and so it is better to use projective coordinates
\begin{equation}\label{eqc1}
X = \frac{x}{x'},\ \ Y = \frac{y - y'}{x'},\ \ x' \text{ and } y'. 
\end{equation}
In this case, $X$ is a boundary defining function for $L$ and $x'$ is a boundary defining function for $\ff$. Since $\beta_0$ is a diffeomorphism in the interior of $X\times_0 X$, it induces a symplectic change of variables 
\beq
(x, y, \xi,\eta, x', y', \xi', \eta')\in T^*(\intx \times \intx) \longmapsto (X, Y, \la,\mu, x', y', \la', \mu') \in T^*(X\times_0 X),
\eeq
given by
\beq
\la = x' \xi,\ \ \mu = x'\eta,\ \ \la' = \xi' + \xi X + \eta Y  \text{ and } \mu' = \eta + \eta'.
\eeq
and  $Q_L$  becomes
\beq
\beta_0^*Q_{L}  = \ha(\tau^2 - X^2( \la^2 + h(x'X, x'Y + y', \mu))),
\eeq
and here we used the fact that $h(x,y,\eta)$ is homogeneous of degree two in $\eta.$

  Away from the face $R$, $\rho_R>\del,$ for some $\del,$ and the function $\log \rho_R$ is smooth. Therefore, as argued above in the case of region 1, the transformation  $(s,m) \mapsto (s+\log \rho_R,m)$ is a $C^\infty$ map away from $\{\rho_R=0\},$ and so it suffices to take
\begin{equation}\label{ch1}
s = t + \log X.
\end{equation}

The change of variable \eqref{ch1}  induces the following symplectic change of variables
\begin{gather}
\begin{gathered}
T^*(\mbr_t \times \intx \times \intx) \longrightarrow  T^*(\mbr_s \times X\times_0 X),\\
(t, \tau, x, y,\xi,\eta, x', y', \xi', \eta')  \longmapsto (s,\sigma, X, Y,\tilde\lambda,\mu, x', y', \la', \mu')\\
\text{where } \tilde \la = \la - \frac{\tau}{X}, \ \ \sigma = \tau,
\end{gathered}
\end{gather}
and the canonical $2$-form on $T^*(\mbr_s\times\xo)$ is given by
\beq
\omega^0 = d\tilde\la\wedge d X + d\mu \wedge dY + d\la' \wedge dx' + d\mu'\wedge dy'.
\eeq
Hence
\beq
\beta_1^* Q_{ L} = - \tilde\la\sigma X - \ha X^2 ({\tilde\la}^2 + h(x'X, x'Y + y', \mu)),
\eeq
and we conclude that
\beq
q_{L} = -\frac{1}{\rho_L}\beta_1^* Q_L=-\frac{1}{X}\beta_1^*Q_L = \tilde\la\sigma + \ha X (\tilde\la^2 + h(x'X, x'Y + y', \mu)).
\eeq
Hence vector field $H_{q_{L}}$ is given by
\begin{gather}
\begin{gathered}
H_{q_L} = \tilde\la \frac{\p}{\p s} + (\sigma + X \tilde\la)\frac{\p }{\p X}  - \ha(\tilde\la^2 + h + x'X\p_X h)\frac{\p}{\p \tilde\la} + \frac{X}{2} H_h + T,
\end{gathered}\label{liftedHA}
\end{gather}
where $T$ is a smooth vector field in $\p_{\la'}, \p_{\mu'}$.   So away from $\sigma=0,$ $H_{q_L}$ is transversal to $\mr_s\times L.$

Next we analyze region 3, near $\mr_s\times (L\cap R)$ and away from $\mr_s\times \ff$.  Here $x, x'$ are boundary defining functions for $\mr_s\times L$ and  $\mr_s\times R$ respectively. In this case, as discussed above, we can take
\beq
s = t + \log x + \log x',
\eeq
which induces the following symplectic change of variable
\begin{gather*}
(t, \tau,x, y,\xi,\eta, x', y' , \xi', \eta') \longmapsto (s,\sigma, x, y, \tilde\xi, \eta, x', y', \tilde\xi', \eta'),\\
\text{where } \tilde \xi =  \xi - \frac{\tau}{x}, \ \ \tilde \xi' = \xi' - \frac{\tau}{x'}, \ \ \sigma = \tau.
\end{gather*}
The symbols can be computed as in the case near $\mr_s\times L$ away from $\mr_s\times \ff$ and $\mr_s\times R$. In particular,
\begin{gather*}
\beta_1^*Q_L=-x\sigma\txi-\ha x^2(\txi^2+h(x,y,\eta))  \text{ and so } q_L=-\frac{1}{\rho_L} \beta_1^* Q_L=\sigma\txi+\ha x(\txi^2+h(x,y,\eta)),\\
\beta_1^*Q_R=-x\sigma\txi'-\ha x'^2(\txi'^2+h(x',y',\eta'))  \text{ and so }  q_R=\sigma\txi'+\ha x'(\txi'^2+h(x',y',\eta')).
\end{gather*}

The Hamilton vector fields are given by
\beq
\bsp
& H_{q_L}= (\sigma+x\txi)\p_x+\txi\p_s +\ha x H_{h(x,y,\eta)}- \ha(\txi^2+h(x,y,\eta)+ x\p_x h(x,y,\eta)) \p_{\txi}, \\
& H_{q_R}= (\sigma+x'\txi')\p_{x'}+\txi'\p_s +\ha x' H_{h(x',y',\eta')}- \ha(\txi'^2+h(x',y',\eta')+ x'\p_{x'} h(x',y',\eta')) \p_{\txi'}.
\end{split}
\eeq 
We conclude that, away from $\sigma=0,$  $H_{q_L}$ is transversal to $\mr_s\times L=\{x = 0\}$ while $H_{q_R}$ is transversal to $\mr_s\times R=\{x' = 0\}.$ 

Finally, we analyze region 4,  near  the co-dimension $3$ corner $\mr_s\times (L\cap \ff\cap R).$  Here we also work with suitable projective coordinates, and without loss of generality, as in \cite{MSV} we may take $\rho_{\ff}=y_1 - y_1' \geq 0$ and take the following coordinates
\begin{equation}\label{eqc2}
u = y_1 - y_1', \ \ w = \frac{x}{y_1 - y_1'}, \ \ w' = \frac{x'}{y_1 - y_1'},\ \ y'  \text{ and } Z_j = \frac{y_j - y_j'}{y_1 - y_1'}, \ \ j = 2, 3,\cdots n.
\end{equation}
Here $w, w'$ and $ u$ are boundary defining functions for $\mr_s\times L,$  $\mr_s\times R$ and $\mr_s\times \ff$ faces respectively.  The induced symplectic change of variables
\begin{gather}
\begin{gathered}
T^*(\intx\times\intx) \longrightarrow  T^*(X\times_0 X) \\ 
(x, y, \xi,\eta, x', y', \xi', \eta') \longmapsto (w,u,Z,\la,\nu,\mu, w', y',  \la', \mu') \\
\text{ where } \\
\la  = \xi u, \ \ \la'  = \xi' u, \ \ \nu = \xi w + \xi'w' + \eta_1 + \sum_{j=2}^n \eta_j Z_j, \\
\mu' = \eta + \eta',  \ \ \mu_j = \eta_j u, \ \ j = 2, 3, \cdots n. 
\end{gathered}\label{eqsym}
\end{gather}

In these coordinates, the symbols of $Q_L$ and $Q_R$ are given by
\begin{gather*}
\beta_0^* Q_{ L}  = \ha( \tau^2 -  w^2(\la^2 + h(uw, y, u\eta))),\\
 \beta_0^* Q_{ R}  = \ha( \tau^2 -  w'^2(\la^2 + h(uw', y', u\mu' - u\eta))),
\end{gather*}
where 
\beq
y = (y_1' + u, y_2' + uZ_2, \cdots, y_{n}' + uZ_{n}),\ \ u\eta = (u\nu - \la w - \la' w' - \sum_{j = 2}^n \mu_j Z_j, \mu).
\eeq
In this case,  we set
\begin{equation}\label{ch2}
s = t + \log w + \log w',
\end{equation}
which  induces the symplectic transformation
\beq
\bsp
& T^*(\mbr_t \times \intx\times\intx) \longrightarrow  T^*(\mbr_s \times X\times_0 X),\\
& (t, \tau,x, y, \xi,\eta, x', y',  \xi', \eta')  \longmapsto (s,\sigma, w,u,Z, \tilde\la, \nu,\mu, w', y', \tilde\la',  \mu') \\
&  \text{where } \tilde \la = \la - \frac{\tau}{w}, \ \ \tilde \la' = \la'- \frac{\tau}{w'}, \ \ \sigma = \tau.
\end{split}
\eeq
Here the canonical $2$-form on $T^*(\mbr_s \times \xo)$ is given by 
\beq
\omega^0 = d\sigma \wedge ds + d\tilde\la \wedge dw + d\tilde\la'\wedge dw' + d\nu\wedge du + d\mu\wedge dZ + d\mu' \wedge d y'.
\eeq
The lifts of the symbols $Q_{L}$ and $Q_{R}$ become
\begin{gather*}
\beta_1^*Q_{ L} = -w\sigma \tilde \la - \ha w^2 (\tilde\la^2 + h(uw, y, \tilde \eta)),\\
 \beta_1^*Q_{ R} = -w'\sigma \tilde\la' - \ha w'^2 ( \tilde\la'^2 + h(uw', y', u\mu' - \tilde \eta)),
\end{gather*}
where $\tilde \eta \doteq u\eta =  (u\nu - \tilde\la w - \tilde\la' w' - 2\sigma - \sum_{j = 2}^n \mu_j Z_j, \mu)$.  Therefore, in these coordinates
\begin{gather*}
q_{L} = -\frac{1}{\rho_L}\beta_1^*Q_L =  \sigma \tilde\la + \frac{1}{2} w(\tilde\la^2 +  h(uw,  y, \tilde \eta)),\\
q_{R} =-\frac{1}{\rho_R}\beta_1^*Q_R=  \sigma \tilde\la' + \frac{1}{2}w '( \tilde\la'^2 +  h(uw', y', u\mu' - \tilde \eta)).
\end{gather*}
Hence the Hamilton vector fields are of the form
\begin{gather}
\begin{gathered}
H_{q_{ L}} = (\sigma + w\tilde\la - \ha w^2 \p_{\tilde\eta}h(uw, y, \tilde \eta))\frac{\p}{\p w} + T_L,\\
  H_{q_{R}} =  (\sigma + w'\tilde\la' + \ha w'^2 \p_{\tilde\eta}h(uw', y', u\mu'- \tilde \eta) )\frac{\p}{\p w'} + T_R,
\end{gathered}\label{VF}
\end{gather}
 where $T_L, T_R$ are smooth vector fields on $T^*(\mbr_s \times X\times_0 X)$ with no $\frac{\p}{\p w},$ $\frac{\p}{\p w'}$ or $\frac{\p}{\p \sigma}$  terms. 
Notice that these vector fields are $C^\infty$ up to the front face, and that away from $\sigma=0,$ the vector $H_{q_L}$ is transversal to $\mr_s\times L$  and $H_{q_R}$ is transversal to $\mr_s\times R.$  This shows that the transversality to $L$ and $R$ holds up to the corner.
This ends the proof of the Lemma.
\end{proof}

Now we  conclude the proof of Theorem \ref{soj}.  Since in the interior of $\xo,$ $\beta_0$ is a $C^\infty$ diffeomorphism between $C^\infty$ open manifolds,  $\beta_0^*\tilde\La$ is a $C^\infty$ Lagrangian manifold in the interior of $\mr_t\times \xo,$ and it is defined as
 \beq
\beta_0^*\tilde\La = \bigcup_{t_1\geq 0, t_2\geq 0} \exp t_2 \beta_0^*H_{Q_{ L}}\circ \exp t_1 \beta_0^*H_{Q_{R}} (\beta_0^*\Sigma),
 \eeq
where
\beq
\Sigma = \{(0, \tau, x,y, \xi,\eta, x, y,-\xi, -\eta) : x^2 \xi^2 + x^2 h(x, y, \eta) = \tau^2\}.
\eeq
  
In projective coordinates
 \begin{gather*}
 x', \;\ X= \frac{x}{x'}, \;\ Y=\frac{y-y'}{x'}, \;\ y',
 \end{gather*}
valid near $\ff$ and $L,$  $\beta_0^*\Sigma$
 can be written as
\beq
\beta_0^*\Sigma = \{(0,1,X, Y,\la,\mu, x', y' , \la', \mu'):  X = 1, Y = 0, \la' = \mu' = 0,  \la^2 + h(x', y', \mu) = \tau^2\},
\eeq
which is a $C^\infty$ submanifold of $T^*(\mr_t\times X\times_0 X)$ that extends smoothly up to the front face $\mr_t\times \ff=\{x'=0\}.$  Since $\beta_0^*\Sigma$ does not intersect either $\mr_s\times L$ or $\mr_s\times R,$ these properties do not change if we set $s=t+\log \rho_R+\log \rho_L,$  and hence
$\beta_1^*\Sigma$  is a $C^\infty$ submanifold of $T^*(\mbr_s \times \xo)$ that has a $C^\infty$ extension up to $\mr_s\times \ff.$   

 In the interior of $\mr_s\times \xo,$ $\beta_1^*Q_{L}$ and $\beta_1^*Q_R$ vanish on $\beta_1^*\tilde\La,$ and hence the integral curves of $H_{q_{L}}$ and 
 $H_{q_R}$ on $\beta_1^*\tilde \La$  coincide with the integral curves of $H_{\beta_1^*Q_{L}}$ and $H_{\beta_1^*Q_R}$ respectively. Therefore, in the interior of $\mr_s\times \xo$ and across to the front face,  $\beta_1^*\tilde\La$ is the union of integral curves of $H_{q_{L}}$ and $H_{q_R}$ emanating from $\beta_1^*\Sigma.$

  Since $Q_R$ and $Q_L$ do not depend on $t,$ it follows that $q_L$ and $q_R$ do not depend on $s,$ and hence $\sigma$ remains constant along the integral curves of $q_L$ and $q_R.$  Since $\sigma=\tau\not=0$ on $\beta_1^* \Sigma,$  it follows that $\sigma\not=0$ on  $\beta_1^*\tilde \La$ in the interior of $\mr_s \times \xo.$ However,  we have also shown that, up to the front face, in the region $\sigma=1,$ $H_{q_L}$ is transversal to $\mr_s\times L$ while $H_{q_R}$ is transversal  up to $\mr_s\times  R.$

  Recall from \eqref{liftedHA} that $H_{q_L}$ and $H_{q_R}$ are $C^\infty$ up to $\mr_s\times \ff$ and are tangent to $\mr_s\times \ff.$   So,  $\beta_1^* \tilde\La$ extends up to  $\mr_s \times \ff$ as the joint  flow-out of $\beta_1^*\Sigma$ by $H_{q_R}$ and $H_{q_L}.$

  So the integral curves of $H_{q_L}$ can be continued smoothly up to $\mr_s\times L$ and
the  integral curves of $H_{q_R}$ can be continued smoothly up to $\mr_s\times R.$   Therefore  $\beta_1^* \tilde{\Lambda}$ can be extended up to the face  $\{\rho_R=0\}$ because $H_{q_R}$ is tangent to $\beta_1^* \tilde{\Lambda}$ and transversal to $\{\rho_R=0\}.$  The same holds for the left face.   This shows that $\beta_1^* \tilde\La,$ which is in principle is defined in the interior of $\mr_s\times \xo,$ extends to a $C^\infty$ manifold up to $\p (\mr_s\times \xo) $ which intersects $\mr_s\times L$ and $\mr_s\times R$ transversally. 

We can make this more precise if we work suitable local symplectic coordinates valid near a point on the fiber over the corner $\ff\cap L \cap R.$   
We know that $R,$ $L$ and $\ff$ intersect transversally.  So one can choose local coordinates $x=(x_1,x_2,x_3, x')$ in $\mr^{2n+2}$ valid near  $\ff\cap L \cap R$ such that
\begin{gather*}
\ff=\{x_3=0\}, \;  R=\{x_1=0\}  \text{ and } L=\{x_2=0\}.
\end{gather*}
and that the symplectic form $\omega^0=d\sigma\wedge ds + d\xi \wedge d x.$  For example, this can be accomplished by using local coordinates defined in \eqref{eqc2} and setting $u=x_3,$ $w=x_2$ and $w'=x_1,$ $(y',Z)=x'.$ 

We know that $\beta_0^*\tilde\La$ is a Lagrangian submanifold of $T^*(\mr_s \times \mr^{2n+2})$ contained in $\{x_1> 0, \; x_2>0, \; x_3\geq 0\},$
 which intersects $\ff=\{x_3=0\}$ transversally.  There are commuting Hamilton vector fields $H_{q_R}$ and $H_{q_L}$ tangent to $\beta_0^*\tilde\La$ that are $C^\infty$ up to $\{x_1=0\}\cup\{x_2=0\}\cup \{x_3=0\},$  and as long as $\sigma\not=0,$ 
$H_{q_R}$ transversal to $R$ and tangent to $L$ and $\ff$ and $H_{q_L}$ is transversal to $L$ and tangent to $R$ and $\ff.$    Also, since $q_R$ and 
$q_L$ do not depend on $s,$ $\sigma$ remains constant along the integral curves of $H_{q_R}$ and $H_{q_L}.$

Let
\begin{gather*}
\mcf=T_{\{x_1=x_2=0\}}^*(\mr_s \times \{x: x_1>0, x_2>0, x_3\geq 0\}),
\end{gather*}
and let  $p=(s,\sigma, 0,\xi_1,0,\xi_2,x_3,\xi_3, x', \xi')),$ $\sigma\not=0,$  denote a point on  $\mcf.$  Since $q_R$ and $q_L$ do not depend on $s,$  $\sigma$ remains constant along the integral curves of  $H_{q_R}$ and $H_{q_L}.$  Moreover, in the region
$\sigma\not=0,$ the vector fields $H_{q_R}$ and $H_{q_L}$ are smooth, non-degenerate up to the boundaries.  $H_{q_R}$ is tangent to $\ff$ and $L,$
while $H_{q_L}$ is tangent to $\ff$ and $R.$ So, for $\eps$ small enough we  define
\begin{gather*}
\Psi_0: [0,\eps)\times [0,\eps) \times (\mcf \cap \{\sigma\not=0\} ) \longrightarrow U_0 \subset T^*(\mr_s\times \{ x_1\geq 0, x_2\geq 0, x_3\geq 0\})\\
\Psi_0(t_1,t_2,p) = \exp(-t_1 H_{q_R})\circ  \exp(-t_2 H_{q_L})(p),
\end{gather*}
and
\begin{gather*}
\Psi_1: [0,\eps)\times [0,\eps) \times (\mcf\cup \{\sigma\not=0\}) \longrightarrow U_1 \subset T^*(\mr_s\times \{ x_1\geq 0, x_2\geq 0, x_3\geq 0\})\\
\Psi_1(t_1,t_2,p) = \exp(-t_1 \p_{x_1})\circ  \exp(-t_2 \p_{x_2})(p),
\end{gather*}
Since the vector fields $H_{q_R},$ $H_{q_L}$ commute and  $\p_{x_1}$ and $\p_{x_2}$ commute, both maps are $C^\infty$ map and moreover, 
\begin{gather*}
\Psi_0^* H_{q_R}=-\p_{t_1}, \;\  \Psi_0^* H_{q_L}=-\p_{t_2}  \\
\Psi_1^* \p_{x_1}=-\p_{t_1}, \;\ \Psi_1^* \p_{x_2}=-\p_{t_2}  \\
\end{gather*}
Hence,
\begin{gather*}
\Psi=\Psi_0\circ \Psi_1^{-1}: U_1 \longrightarrow U_0, \\
\Psi^* H_{q_R}=-\p_{x_1},  \;\  \Psi^* H_{q_L}=-\p_{x_2}.
\end{gather*}
Moreover,  if $\omega^0$  is the symplectic form on $T^*(\mr \times \xo),$ in coordinates \eqref{eqc2} valid near $\mcf,$
\begin{gather*}
\Psi^*\omega^0= \omega^0.
\end{gather*}

Now $\Upsilon=\Psi^{-1}(\beta_0^*\tilde\La)$ is a $C^\infty$ Lagrangian in $\{ x_1> 0, \; x_2>0, x_3\geq 0\}$ which intersects $\{x_3=0\}$ transversally, and
both $\p_{x_1}$ and $\p_{x_2}$ are tangent to $\Upsilon.$  But this implies that  for any point $p \in \Upsilon,$ the integral curves of
 $\p_{x_j},$ $j=1,2$ starting at a point $p\in \Upsilon$ are contained in $\Upsilon.$  Therefore, for any $p=(x_1,\xi_1, x_2,\xi_2, x_3,\xi_3,x', \xi')\in \Upsilon,$ with 
 $x_1$ and $x_2$ small enough, the set
$\{x_1-t_1,\xi_1, x_2-t_2,\xi_2, x_3,\xi_3, x',\xi'\} \subset \Upsilon.$  By taking $t_1$ and $t_2$ large enough, this gives an extension $\overline{\Upsilon}$ of $\Upsilon$ to $\{x_1\leq 0\}\cup \{x_2\leq 0\}.$ Now $\Psi(\overline{\Upsilon})$ is the desired Lagrangian extension of $\beta_0^*\tilde\Lambda.$  Notice that in fact, it extends past the boundaries $\{x_1=0\}$ and $\{x_2=0\}.$  The construction in the other regions, away from the  co-dimension three corners follows by the same argument.

We still need  to verify that $\La^* \cap T_{\{\rho_\bullet=0\}}^*( \xo)$  is a $\CI$ Lagrangian submanifold of $T^*\{\rho_\bullet=0\}.$   To see that, observe that we
have constructed  local symplectic coordinates $(x,\xi)$ near a point $p\in  \La^*\cap \{\rho_R=\rho_L=0\}$ such that $R=\{x_1=0\}$ and $L=\{x_2=0\}$ and  $\Psi^*H_{q_R}=H_{\xi_1}=\p_{x_1}$ and $\Psi^*H_{q_L}=H_{\xi_2}=\p_{x_2}.$ Therefore $\Psi^*q_R= \xi_1+C_1$ and $\Psi^*q_L=\xi_2+C_2.$   But since $q_R(p)=q_L(p)=0,$ it follows that $\xi_1(p)=\xi_2(p)=0,$  and so  $C_1=C_2=0.$   So $\xi_1=\xi_2=0$ on $\Upsilon.$ But  $\Upsilon$ is foliated by submanifolds
\begin{gather*}
\Upsilon_{a}= \Upsilon \cap \{x_1=a\}, \;\  \Upsilon^{a}= \Upsilon \cap \{x_2=a\}, \;\ 
\end{gather*}
which are Lagrangian submanifolds of $T^*\{x_j=a\},$ $j=1,2$ because $\xi_j=0$ on $\Upsilon.$  In particular this shows that $\Upsilon_0= \La^* \cap \{\rho_R=0\}\subset T^*\{\rho_R=0\}$ and $\beta^*\wtla \cap \{\rho_L=0\}\subset T^*\{\rho_L=0\}$  are Lagrangian submanifolds. The same argument shows that and 
$\La^*\cap \{\rho_R=\rho_L=0\}\subset T^*\{\rho_R=\rho_L=0\}$ is a Lagrangian submanifold. 
\end{proof}

As in the notation of Theorem \ref{soj}, we shall  denote
\begin{gather}
\begin{gathered}
\wtla_\pm^f \text{ to be the extension of } \beta_{1f}^* \La_{\pm}  \text{ up to } \p T^*(\mr\times \xo).
\end{gathered}\label{deflastar}
\end{gather}

We know from Theorem \ref{soj} that
\begin{gather*}
\p_L \wtla_{\pm}^\bullet = \wtla_\pm^\bullet \cap \{\rho_L=0\} \text{ and } \p_R \La_\pm^\bullet = \wtla_\pm^\bullet \cap \{\rho_R=0\}, \;\ \bullet=f,b,
\end{gather*}
are $C^\infty$ closed Lagrangian submanifolds of  $T^* (\mr_s\times (\p X \times_0 X) )\setminus 0$ and $T^* (\mr_s\times ( X \times_0 \p X) )\setminus 0$ respectively,  in the sense that they can be extended to a $C^\infty$ manifold across the boundary of $T^* (\mr_s\times (\p X \times_0 X))$ or $T^* (\mr_s\times ( X \times_0 \p X)).$ 

 We also define
\begin{gather}
\begin{gathered}
\La_{\p\pm}^{\bullet}= \wtla_\pm^\bullet \cap \{\rho_R=\rho_L=0\},
\end{gathered}\label{bdryla}
\end{gather}
and these are Lagrangian submanifolds of $T^*(\mr\times \p X\times_0 \p X).$   Away from the front face of $\mr\times \p X \times_0 \p X,$ each one of the manifolds $\La_{\p\pm}^f$ defines a canonical relation on $\mr \times \p X \times \p X \times \mr$ as the set
\begin{gather}
\begin{gathered}
\{(s,\sigma, y,\eta); (y',\eta', s', \sigma): \text{ such that  there exists an integral curve of }  H_{q_{fR}} \text{ joining } \\  (y',\eta',s',\sigma) \text{ and } (z, -\zeta, s=\log x(z), \sigma) \text{ and }
\text{ an integral curve of }  H_{q_{fL}} \\ \text{ joining } (z, -\zeta, s=\log x(z), \sigma) \text{ and }  (s,\sigma,y,\eta), \;\ \mp \sigma>0  \},
\end{gathered}\label{relation}
\end{gather}
and therefore we define $\La_\p^f=\La_{\p+}^f\cup \La_{\p-}^f$ to be the scattering relation of a non-trapping AHM $(\intx,g).$  As mentioned in the introduction, we could have defined $\La_{\p}^f\cap\{\sigma=-1\}$ to be the scattering relation because  $\La_\p^f$ is foliated by $\La_\p^f \cap\{ \sigma=\text{constant}\},$ and according to \eqref{relation} the leaf with $\sigma=-1$ corresponds to the bicharacteristics that project onto  unit speed geodesics. The other leaves  of \eqref{relation} are  associated with reparametrized integral curves of the Hamilton vector fields.

\section{Proof of Theorem  \ref{scat-FIO}}

We shall use H\"ormander's notation: $v \in I^m(Y,\La, \Omega_Y^\ha)$ denotes a half-density valued Lagrangian distribution $v$ on the manifold $Y$ of order $m$ with respect to the Lagrangian $\La\subset T^* Y\setminus 0,$ where $Y$ is a $C^\infty$ manifold.

We assume that $(\intx,g)$ is a non-trapping AHM. Let $\mcw= A^{-1}\sin(tA),$ $A=\sqrt{\Delta_g-\nsq},$ and let $\mcw=\mcw_++\mcw_-,$ with $\mcw_\pm$ defined at the beginning of Section \ref{ULM}.

The following result gives a uniform description of $\beta_{1f}^* (xx')^{-\novt}\mck_{\mcw(t)}$ up to the boundary faces of $\xo,$ and therefore gives a thorough  description of the limit \eqref{kerscat-0}.  

\begin{prop}\label{radfs-th} Let  $K_{\mcw_\pm }\in C^{-\infty}(\mr_t \times  \intx \times \intx)$  denote the Schwartz  kernels of $\mcw_\pm.$  Then $\beta_{1f}^*((xx')^{-\novt}K_{\mcw_\pm})$ has an extension up to the boundary of $\mr \times \xo$ such that
\begin{gather}
\begin{gathered}
\beta_{1f}^*( (xx')^{-\novt}  K_{\mcw_\pm} ) \in I^{-\fq}(\mr_s \times \xo; \wtla_{\pm}^f, \Omega_{\mr_s \times \xo}^\ha), \\
\beta_{1f}^*( (xx')^{-\novt}  K_{\mcw_\pm })|_{\{\rho_L=0\}} \in  I^{-1}(\mr_s \times \p \xo; \p_L\wtla_{\pm}^f, \Omega_{\mr_s \times \p \xo}^\ha), \\
\beta_{1f}^*( (xx')^{-\novt}  K_{\mcw_\pm })|_{\{\rho_R=0\}} \in  I^{-1}(\mr_s \times X\times_0 \p X; \p_R \wtla_\pm^f, \Omega_{\mr_s \times X\times_0 \p X}^\ha), \\
\beta_{1f}^*( (xx')^{-\novt}  K_{\mcw_\pm })|_{\{\rho_R=\rho_L=0\}} \in  I^{-\tq}(\mr_s \times \p X\times_0 \p X;  \La_{\p\pm}^f, \Omega_{\mr_s \times \p X\times_0 \p X}^\ha).
\end{gathered}\label{order-Lag}
\end{gather}
\end{prop}
\begin{proof} 
Since $\beta_1$ is a diffeomorphism in the interior of $\mr_s\times \xo,$   it follows that, away from $\diag_0,$ 
 \begin{gather}
 \begin{gathered}
\beta_{1f}^* ({(xx')}^{-\novt} \mck_{\mcw_+})  \in I^{-\fq}(\mr_s\times \xo, \beta_{1f}^*\La_+, \Omega_{\mr \times \xo}^\ha), \\
 \beta_{1f}^* ({(xx')}^{-\novt} \mck_{\mcw_-} ) \in I^{-\fq}(\mr_s\times \xo, \beta_{1f}^*\La_-, \Omega_{\mr \times \xo}^\ha),  
 \end{gathered}\label{lift-lag-int}
 \end{gather}
  and one would like to extend this regularity  up to the left boundary of $\mr_s\times \xo.$

  We have shown in Theorem \ref{soj} that   $\beta_{1f}^*\wtla_\pm$  can be extended smoothly up across the boundary and intersects the boundary transversally.  So one expects that  $\beta_{1f}^* ({(xx')}^{-\novt}\mck_{\mcw_\pm})$  can be extended across the boundary of $\mr_s\times \xo,$ and to do this one needs to analyze the behavior of the symbol of these distributions up to the boundary. We work with $\beta_{1f}^*({(xx')}^{-\novt}\mck_{\mcw_+}),$ the other case is identical, and we prove the result on each face separately. 

We work with the operator $x^{\novt} \square_L x^{-\novt}.$  The transport equation for the  principal symbol  of $(xx')^{-\novt} \mck_{\mcw_+}$ is given by
\begin{gather}
(\mcl_{H_{Q_L}}+c) a=0,  \text{ on } \La_+ \label{transp1}
\end{gather} 
 where  $\mcl_{H_{Q_L}}$ is the Lie derivative with respective to the vector field $H_{Q_L},$ and $c$ is subprincipal symbol of  $x^{-\novt} \square_L x^\novt $   Since the manifold is non-trapping, the symbol is well defined in $T^*(\mr \times \intx \times \intx).$   In the interior of $T^*(\mr \times \xo),$ the map $\beta_0$ is a diffeomorphism and  in view of \eqref{shift} the manifold $\beta_{1f}^*\La_+$ is obtained from $\beta_0^*\La_+$ by the map $(t,\tau, m,\nu)\longmapsto (s,\sigma, \nu+d\gamma).$ Again, in the interior this is a $C^\infty$ symplectic change of variables, and therefore  \eqref{transp1} becomes
 \begin{gather}
 (\mcl_{H_{\beta_{1f}^*Q_{L}}}+ \beta_{1f}^*c) \beta_{1f}^* a=0 \text{ on } \beta_{1f}^* \La_+, \text{ in the interior of } T^*(\mr \times \xo). \label{transp2}
 \end{gather}
We want to show this equation can be solved up to the boundary.  But near a point $\alpha \in \beta_{1f}^* \La_+\cap \{\rho_L=0\},$ the manifold $\beta_{1f}^*\La_+$ can be parametrized by a $C^\infty$ phase function $\Phi(m,\theta),$ and therefore  $ \beta_{1f}^*( {x}^{-\novt} \mck_{\mcw_+})$ is  microlocally given by an oscillatory
 \begin{gather}
 \beta_{1f}^*( {x}^{-\novt} \mck_{\mcw_+})= \int_{\mr^N} e^{i \Phi(m,\theta) } a(m,\theta) \; d\theta, \label{osc}
 \end{gather}
 where 
 \begin{gather*}
 a(m,\theta)\sim a_0(m,\theta)+ a_1(m,\theta) +..., \; a_j(m,\theta) \in S^{-\fq + \frac{(2n+2-2N)}{4} -j}, \; j=0,1,2,\ldots
 \end{gather*}
  is a $C^\infty$ symbol in $\rho_L>0.$  The term $2n+2$ is the dimension of $\xo.$ We will show that $a$ extends smoothly up to $\{\rho_L=0\}.$  We will do the computation in the region near $L$ and away from $R\cup \ff,$  and the other cases are left to the reader. The computations are  very similar to the ones done in the proof of Theorem \ref{vfields}.
  
  As in the proof of Proposition  \ref{vfields}, in the region near $L$ are away from $R \cup \ff,$ we can just use $x$ such that \eqref{prod} is defined as the defining function of $L$ and set $s=t+\log x.$  In these coordinates the Laplacian is given by
\begin{gather*}
\Delta_g= (xD_x)^2- in xD_x -ix^2A D_x + x^2 \Delta_h,
\end{gather*}
and therefore
\begin{gather}
x^{-\novt} (D_t^2-\Delta_g+\nsq) x^{\novt}= D_t^2- (xD_x)^2 -x^2\Delta_h+ ix^2A D_x -\novt xA.\label{conjwe}
\end{gather}
and therefore
\begin{gather*}
\beta_{1f}^*( x^{\novt} (D_t^2-\Delta_g+\nsq) x^{-\novt})=x Q_{fL}(x,y,D_x,D_y), \text{ where } \\
Q_{fL}= x D_x^2+ 2D_x D_s+ x \Delta_h -i(1- xA) D_x + iA D_s- \novt A.
 \end{gather*}

 If we apply $Q_{fL}(x,y,D_x,D_y)$ to  \eqref{osc} we find that
 \begin{gather*}
 Q_{fL} \int_{\mr^N} e^{i \Phi(m,\theta) } a(m,\theta) \; d\theta= \int_{\mr^N} e^{i \Phi(m,\theta) } b(m,\theta) \; d\theta, \text{ where } \\
 b(m,\theta)\sim \sum_{j=0}^\infty b_j(m,\theta), \;\  b_0(m,\theta)= (H_{q_{fL}} + Q_{fL}\Phi) a_0, \;\ b_j=  (H_{q_{fL}} + Q_{fL}) a_j+ Q_{fL} a_{j-1},
 \end{gather*}
 and since $H_{fL}$ is transversal to $L,$ we obtain  a sequence of $C^\infty$ non-degenerate transport equations that can be solved up to $\{x=0\},$ namely
 \begin{gather*}
 (H_{q_{fL}} + Q_{fL}\Phi) a_0=0, \;\ a_0 \in C^\infty \text{ for } x>\eps, \\
 (H_{q_{fL}} + Q_{fL}\Phi) a_j+ Q_{fL} a_{j-1}=0, \;  a_j \in C^\infty \text{ for } x>\eps.
 \end{gather*}
 
  So we conclude that away from $\diag_0,$
  \begin{gather}
  \begin{gathered}
\beta_{1f}^* ({(xx')}^{-\novt} \mck_{\mcw_+})  \in I^{-\fq}(\mr_s\times \xo,\wtla_+^f,\Omega_{\mr\times \xo}^\ha), \\
 \beta_{1f}^*({(xx')}^{-\novt} \mck_{ \mcw_-})  \in I^{-\fq}(\mr_s\times \xo, \wtla_-^f,\Omega_{\mr\times \xo}^\ha). 
\end{gathered}\label{lift-lag-clos}
 \end{gather}
 
 Next we consider the projections to $\p\xo,$ $X\times_0 X$ and $\p X \times_0 \p X.$   Since,  as explained above,  $\La_{\pm}^f$ intersects  $\{\rho_\bullet=0\},$ $\bullet=R,L,$  transversally and  at $\{\xi_\bullet=0\},$ where $\xi_\bullet$ is the dual to $\rho_\bullet,$  if a phase $\Phi(m,\theta)$ parametrizes $\La_{\pm}^f$ near a point 
 $\alpha\in T_{\p\xo}^*\xo,$ then $\Phi|_{\{\rho_L=0\}}$  parametrizes $\p_L \La_{\pm}^f.$ Similarly, if $\Phi(m,\theta)$ parametrizes $\La_{\pm}^f$ near a point 
 $\alpha\in T_{X\times _0 \p X}^*\xo,$  then $\Phi|_{\{\rho_R=0\}}$  parametrizes $\p_R \La_{\pm}^f,$  and if  $\Phi(m,\theta)$ parametrizes $\La_{\pm}^f$ near the corner a point on the fiber over the corner  $\{\rho_R=\rho_L=0\},$ then 
 $\Phi|_{\{\rho_R=\rho_L=0\}}$  parametrizes $\La_{\pm,\p}^f.$  This just follows from the definition of parametrization of a Lagrangian by a phase function, but it can be found in 
 Proposition 4.1.7 of \cite{Hormander-FIO}.   Here one strongly needs that $\La_{\pm}^f$ intersects  $\{\rho_\bullet=0\},$ $\bullet=R,L,$  transversally and  at $\{\xi_\bullet=0\},$ where 
 $\xi_\bullet$ is the dual to $\rho_\bullet.$
 
 Therefore, if  
 \begin{gather*}
 \beta_{1f}^*((xx')^{-\novt}\mcw_+)= \int_{\mr^N} e^{i \Phi(m,\theta)} a(m,\theta) \; d\theta \text{ then } \\
  \beta_{1f}^*((xx')^{-\novt}\mcw_+)|_{\{\rho_\bullet=0\}}= \int_{\mr^N} e^{i \Phi_\bullet(m,\theta)} a_\bullet(m,\theta) \; d\theta,\\ \text{ where } \Phi_\bullet(m,\theta)= \Phi(m,\theta)|_{\{\rho_\bullet=0\}},\;\ a_\bullet(m,\theta)= a(m,\theta)|_{\{\rho_\bullet=0\}}, \;\ \bullet=R,L, \\
\beta_{1f}^*((xx')^{-\novt}\mcw)|_{\{\rho_R=\rho_L=0\}}= \int_{\mr^N} e^{i \Phi_\p(m,\theta)} a_\p(m,\theta) \; d\theta, \text{ and where } \\
\Phi_\p(m,\theta)= \Phi(m,\theta)|_{\{\rho_R=\rho_L=0\}},\;\ a_\p(m,\theta)= a(m,\theta)|_{\{\rho_R=\rho_L=0\}}.
\end{gather*}
  
 As for the order of the operators, one sees that the number of variables $\theta$ stays the same, while the dimension drops by one for the projection to $\{\rho_\bullet=0\},$ 
 $\bullet=R,L$ and by two for the projection to $\{\rho_R=\rho_L=0\}.$  Since $a\in S^{-\fq+ \frac{2n+2-N}{4}}= S^{-1+ \frac{2n+1-N}{4}}= S^{-\tq+ \frac{2n-N}{4}}.$ This proves \eqref{order-Lag} and concludes the proof of Proposition \ref{radfs-th}.
 \end{proof}

Now we are in a position to conclude the proof of Theorem  \ref{scat-FIO}.
\begin{proof}  The first step is to interpret \eqref{kerscat-0} in terms of the blow-ups defined above.   We claim that, the case of non-trapping AHM manifolds,  the limit \eqref{kerscat-0} holds in a stronger sense than in \eqref{sep-limits}.  In fact, we have 

\begin{gather}
\begin{gathered}
\lim_{x\rightarrow 0} \lim_{x'\rightarrow 0} {(xx')}^{-\novt}  E_+(s-s'-\log x-\log x', z, z')= \\ \beta_{\p*} \left( \beta_{0}^* ((xx')^{-\novt} E_+)(s-s'-\log \rho_L-\log \rho_R-2\log \rho_{\ff}, m)|_{\{\rho_R=\rho_L=0\}}\right). 
\end{gathered}\label{lift-limit}
\end{gather}
In other words, the kernel of $(xx')^{-\novt}E_+$ is pulled back by  $\beta_{1f}$, projected to $\{\rho_R=\rho_L=0\}$ and then pushed forward to 
$\mr \times \p X \times \p X \times \mr$ by $\beta_\p.$ Since one is taking the limit in $x'$ first and then in $x,$ and since $x=\rho_{\ff}\rho_L,$ $x'=\rho_{\ff}\rho_R,$ it follows that $\rho_{\ff}\not=0.$ So one is in fact restricting to the right face, and then  to the left face, which is exactly \eqref{lift-limit}. Then it follows from \eqref{kerscat-0} and \eqref{lift-limit} that if $\mck_{\mcs}$ denotes the kernel of the scattering matrix, then
\begin{gather}
\beta_\p^*(\mck_{\mcs})(s,\tilde m, s')=  \ha \beta_{0}^*((xx')^{-\novt} E_+)(s-s'-\log \rho_L-\log \rho_R-2\log \rho_{\ff}, m)|_{\{\rho_R=\rho_L=0\}}. \label{scat-JS0}
\end{gather}

If we denote 
\begin{gather*}
\beta_{0}^* ((xx')^{-\novt} E_+)(s-\log \rho_L-\log \rho_R, m)|_{\{\rho_R=\rho_L=0\}}=\mca(s,\tilde m),
\end{gather*}
then it follows from Proposition \ref{radfs-th} that  $\mca$ satisfies \eqref{kerscatmat1}, and it follows from \eqref{scat-JS0} that
\begin{gather*}
\beta_\p^*\mck_{\mcs}(s,\tilde m, s')= \mca(s-s'-2\log \rho_{\ff_0},\tilde m).
\end{gather*}

This ends the proof of Theorem \ref{scat-FIO}.
\end{proof}

The analogue of this formula in terms of the scattering matrix and the resolvent was established in \cite{JS1}, see also \cite{GZ,GrZw}, and it shows that the lift of the Schwartz kernel of $A(\la)$ by $\beta_\p$ can be obtained from the Schwartz kernel of the (forward) resolvent  by 
\begin{gather}
\beta_\p^* A(\la)= 2i\la \beta_0^*\left((xx')^{-\novt-i\la} R_+(\la)\right)|_{\{\rho_R=\rho_L=0\}}. \label{scatmat-JS}
\end{gather}
But as we know from \eqref{res-fw} that $\mcr_+(\la,z,z')=\widehat{E_+}(\la,z,z'),$ and thus \eqref{scatmat-JS} is in some sense the Fourier transform of \eqref{kerscat-0}.
 We know from Lemma \ref{conv-FT} that we can commute the Fourier transform and the projection to either $\{\rho_R=0\}$ or $\{\rho_L=0\},$ but we cannot  show directly that the Fourier transform commutes with the second restriction.

Finally we remark that Proposition \ref{radfs-th} can be used to say more about the microlocal structure of the Schwartz kernel of the radiation fields.  Recall that
\begin{gather*}
\mcr_+(f_1,f_2)= \lim_{x\rightarrow 0}  x^{-\novt} (\p_t \one_+u)(s-\log x, x,y),
\end{gather*}
where $u(t,z)$ is a solution of \eqref{CPWE}.   Then $\mck_{\mcr_+},$ the Schwartz kernel of $\mcr_+$ is an element of $C^{-\infty}(\mr\times \p X \times \intx)$ and as above we want to analyze the microlocal structure of $\beta_{L}^*\mck_{\mcr_+}.$   In view of what was said above, it is enough to analyze the behavior of $\beta_0^*\mck_{\mcw}(s-\log x-\log \rho_R, m)|_{\{\rho_L=0\}}.$ So the following is a direct consequence of Proposition \ref{radfs-th}:
\begin{theorem}\label{radfs-th-N} Let $(\intx,g)$ be a non-trapping AHM.  Let $\mck_{\mcr_+}$ denote the Schwartz kernel of the forward radiation field.   Let $\beta_{1L}$ be the projection $\beta_1|_{\{\rho_L=0\}}.$ Then $\beta_{1L}^* \mck_{\mcr_+}= \mcd(s-\log \rho_{\ff_L},m)+\p_s \mcd(s-\log \rho_{\ff_L},m),$ where  $\rho_{\ff_L }$ is a boundary defining function of the front face of $\p\xo$ and 
\begin{gather*}
\mcd \in  I^{0}(\mr_s \times \p \xo; \p_L\wtla_+^f, \Omega_{\mr_s \times \p \xo}^\ha)+  I^{0}(\mr_s \times \p \xo; \p_L\wtla_-^f, \Omega_{\mr_s \times \p \xo}^\ha).
\end{gather*}
\end{theorem}
S\'a Barreto and Wunsch \cite{SW} proved that $\mck_{\mcr_+}$ is a Lagrangian distribution in $\mcr\times \p X \times \intx,$ which is in essence Theorem \ref{radfs-th-N} in the region away from the front face of $\p X \times_0 X$ and the lift of $\p X \times \p X.$ Theorem \ref{radfs-th-N} gives a uniform version of the result of \cite{SW} up to the front face of  $\p X \times_0 X$ and the lift of $\p X \times \p X$ by $\beta_{0L}.$

\section{Acknowledgements}
 
 The first author visited the Mathematics Department of the Universidade Federal de Santa Catarina at Florian\'opolis  (UFSC)  during the month of May, 2016  when part of this work was done.  His visit to UFSC was supported by a grant of Professor Visitante Especial from CAPES, Brazil.  The work was also supported by a grant from the Simons Foundation (\#349507, Ant\^onio S\'a Barreto). 
%
%===============================REFERENCES==========================================%

\end{document}